\documentclass[11pt]{amsart}
\usepackage[centering,textheight=615pt,textwidth=360pt]{geometry}

\usepackage{color,enumerate}

\newtheorem{theorem}{Theorem}[section]
\newtheorem{lemma}[theorem]{Lemma}
\newtheorem{proposition}[theorem]{Proposition}
\newtheorem{corollary}[theorem]{Corollary}

\theoremstyle{definition}

\newtheorem{remark}[theorem]{Remark}

\numberwithin{equation}{section}

\newcommand{\R}{\mathbf{R}}

\newcommand{\cB}{\mathcal{B}}
\newcommand{\cI}{\mathcal{I}}
\newcommand{\cJ}{\mathcal{J}}

\newcommand{\cM}{\mathcal{M}}
\newcommand{\cE}{\mathcal{E}}
\newcommand{\cQ}{\mathcal{Q}}

\newcommand{\eps}{\varepsilon}
\newcommand{\loc}{{\rm loc}}

\DeclareMathOperator{\supp}{\operatorname{supp}}
\DeclareMathOperator{\card}{\operatorname{card}}

\begin{document}

\title[Critical norm blow-up]{Critical norm blow-up 
for the energy supercritical nonlinear heat equation}

\author[H. Miura]{Hideyuki Miura}
\address{Department of Mathematics, 
Tokyo Institute of Technology, Tokyo 152-8551, Japan}
\email{hideyuki@math.titech.ac.jp}

\author[J. Takahashi]{Jin Takahashi}
\address{Department of Mathematical and Computing Science, 
Tokyo Institute of Technology, Tokyo 152-8552, Japan}
\email[Corresponding author]{takahashi@c.titech.ac.jp}

\begin{abstract}
We address the critical norm blow-up problem 
for the nonlinear heat equation 
$u_t-\Delta u=|u|^{p-1}u$ in $\R^n\times(0,T)$. 
In the supercritical range $p>(n+2)/(n-2)$, 
we prove that if the maximal existence time $T$ is finite, 
then $\lim_{t\to T}\|u(\cdot,t)\|_{L^{n(p-1)/2}(\R^n)} =\infty$ 
without assuming extra conditions such as radial symmetry 
or the type of blow-up. 
\end{abstract}

\maketitle

\section{Introduction}
We study blow-up solutions of the following 
nonlinear heat equation: 
\begin{equation}\label{eq:main}
\left\{
\begin{aligned}
	&u_t-\Delta u=|u|^{p-1}u &&\mbox{ in }\R^n\times(0,T), \\
	&u(\cdot,0)=u_0 &&\mbox{ in }\R^n, 
\end{aligned}
\right.
\end{equation}
where $p>1$ and $n\geq1$. 
This equation has an invariance under 
the scaling $u(x,t)\mapsto \lambda^{2/(p-1)} u(\lambda x, \lambda^2 t)$ 
for $\lambda>0$ and is referred as one of the most typical model of 
scaling invariant scalar nonlinear parabolic equations. 
It is known \cite{BC96,We80} that a local-in-time classical solution
exists uniquely in the scaling critical space $L^{q_c}(\R^n)$ if $q_c > 1$, where 
\[
	q_c:=\frac{n(p-1)}{2}. 
\]

If the maximal existence time $T$ is finite, 
the blow-up occurs in the sense that 
$\lim_{t\to T}\|u(\cdot,t)\|_{L^\infty(\R^n)}=\infty$. 
It was shown \cite{BC96,We80} that $L^q$ norms also blow up for $q>q_c$. 
For $q<q_c$, blow-up solutions with bounded $L^q$ norms 
exist \cite{FM85}. 
These lead to the critical norm blow-up problem: 
If $T<\infty$, does 
\begin{equation}\label{eq:CNB}\tag{CNB}
	\lim_{t\to T}\|u(\cdot,t)\|_{L^{q_c}(\R^n)} =\infty
\end{equation}
hold? 
The problem was stated in Brezis and Cazenave \cite[Open problem 7]{BC96} 
and was unsolved except for some particular cases. 
Recently, in the significant work of Mizoguchi and Souplet \cite{MS19}, 
they discovered that 
the type~I blow-up implies \eqref{eq:CNB} for all $p>1$, where
the blow-up is of type~I if 
$\limsup_{t\to T} (T-t)^{1/(p-1)}\|u(\cdot,t)\|_{L^\infty(\R^n)}<\infty$ 
and type~II if it is not of type~I. 
By \cite{GK87,GMS04}, the blow-up is type~I 
in the Sobolev subcritical range $p<p_S$, 
and thus \eqref{eq:CNB} always holds for $p<p_S$. 
Here $p_S:=\infty$ for $n=1,2$ and 
\[
	p_S:=\frac{n+2}{n-2} \quad \mbox{ for }n\geq3. 
\]
In the critical case $p=p_S$, 
several counterexamples are known. 
In particular, for $3\leq n\leq 5$, 
type~II blow-up solutions with bounded critical norm were constructed 
in \cite{dPMW19,dPMWZZpre,Sc12}.

As for the supercritical range $p> p_S$, although
we have confirmed \cite{MTpre} that the weaker property 
\[
	\limsup_{t\to T}\|u(\cdot,t)\|_{L^{q_c}(\R^n)} =\infty 
\]
holds for general blow-up solutions, 
it remains open that whether the original statement \eqref{eq:CNB} holds or not. 
The goal of this paper is to give the following complete answer.

\begin{theorem}\label{th:main}
Let $n\geq3$, $p>p_S$ and $u$ be a classical solution 
of \eqref{eq:main} with $u_0\in L^{q_c}(\R^n)$. 
If the maximal existence time $T>0$ is finite, then 
\[
	\lim_{t\to T}\|u(\cdot,t)\|_{L^{q_c}(\R^n)}=\infty. 
\]
\end{theorem}

The behavior of the critical norms near
the blow-up time is extensively studied 
for various equations. 
Among others, 
for the three dimensional Navier-Stokes equations (NS), 
Escauriaza, Seregin and \v{S}ver\'{a}k \cite{ESS03} 
showed the unboundedness of the critical norm, 
that is, $\limsup_{t\to T} \|u(\cdot,t)\|_{L^3}=\infty$. 
Seregin \cite{Se12} proved that 
$\lim_{t\to T} \|u(\cdot,t)\|_{L^3}=\infty$, 
the blow-up of the critical norm. 
However, especially in the nonradial setting, 
only few works addressed the blow-up 
of the critical norm as in \cite{Se12} rather than the unboundedness. 
The main problem is the possibility of oscillation for the critical norm. 
Such a problem is standard and difficult 
as remarked by Merle and Rapha\"el \cite{MR08} 
in the context of the nonlinear Schr\"odinger equation.

We prove Theorem \ref{th:main} by a blow-up analysis 
under the contradiction assumption 
$\liminf_{t\to T} \|u(\cdot,t)\|_{L^{q_c}(\R^n)}<\infty$. 
Compared our proof here with \cite{Se12}, 
there are striking differences in 
(i) the compactness procedure
and (ii) the characterization of the blow-up limit. 
The approach stated below would be applied 
to the other scaling invariant parabolic equations.
We hope to address this subject in future works.

(i) The strong compactness for (NS) in $L^3$ 
can be easily guaranteed by the local energy inequality, 
even when any assumptions on the critical norm are not imposed.
For \eqref{eq:main}, due to the absence of derivatives in the nonlinear term,
the corresponding local energy inequality 
does not imply such a compactness, at least directly. 
Moreover, we could not directly appeal general compactness theorems, 
such as the Aubin--Lions theorem, due to the lack of the uniform boundedness 
in the contradiction assumption. 
We overcome these difficulties by studying the defect measure 
in the weak limit of the rescaled solutions 
motivated by Lin and Wang \cite[Chapters 8, 9]{LWbook}, 
where they studied obstructions for
the global regularity of weak solutions for the harmonic map heat flow. 
We note that, very recently, 
the analysis of the defect measure 
was also successfully applied to an important work of Wang and Wei \cite{WW21pre}
for positive solutions of the equation \eqref{eq:main} with $p=p_S$.

(ii) In \cite{Se12}, 
a certain spatial decay estimate of the local energy for (NS) 
plays a key role to ensure the regularity of the blow-up limit in the far field.
Unfortunately, we could not derive such a useful decay estimate for (1.1).
Instead, we show the self-similarity of the blow-up limit 
from the monotonicity for the Giga--Kohn weighted energy \cite{GK85}. 
Then by our contradiction assumption with its scaling invariance,
we can rule out the oscillation and 
show the decay at the spatial infinity in the sense of local integrals.

In addition to the above differences, 
to obtain the nondegeneracy of the blow-up limit, 
we show an $\eps$-regularity theorem 
under the smallness of the scaled $L^{p+1}$ norm. 
This is an essential refinement of the previous $\eps$-regularity in \cite{CDZ07,MTpre} 
under assuming the smallness of the scaled $L^{p+1}$ and $\dot{H}^1$ norms.

In the rest of this section, we give several remarks 
and state the organization of this paper. 
For further discussion, 
see the remarks in \cite[Subsection 1.2]{MTpre} 
including more comprehensive references.

\begin{remark}[Class of solutions]\label{rem:defsol}
It is well-known \cite{BC96,We80} that 
if $q_c>1$, then for any $u_0\in L^{q_c}(\R^n)$, 
\eqref{eq:main} admits a unique classical 
solution in the class 
$u\in C([0,T); L^{q_c}(\R^n))\cap L^\infty_\loc((0,T);L^\infty(\R^n))$. 
For $q_c>p$, unconditional uniqueness \cite{BC96} also holds 
in $C([0,T); L^{q_c}(\R^n))$. 
We note that $q_c>p$ is automatically satisfied when $p>p_S$. 
\end{remark}

\begin{remark}[Critical Lorentz norm]
In Theorem \ref{th:main}, the Lebesgue norm $L^{q_c}(\R^n)$ can be 
replaced by the Lorentz norm $L^{q_c,r}(\R^n)$ with $r<\infty$, 
but cannot be replaced by $L^{q_c,\infty}(\R^n)$. 
For $p>p_S$ with $3\leq n\leq 10$ and 
$p_S<p<p_L:=(n-4)/(n-10)$ with $n\geq 11$, 
by \cite[Theorem 1.1 (ii)]{Mi02}, there exists 
a radial incomplete blow-up solution with a compactly supported $u_0\geq0$. 
By \cite[Remark 1.16]{MM04}, \cite[Proposition 2.7]{MM09} and 
\cite[Proposition C.1]{MM11}, 
the solution satisfies 
\begin{align}
	&0\leq u(x,t)\leq C|x|^{-\frac{2}{p-1}}, 
	\quad x\in \R^n, 0\leq t<T, \label{eq:unif2p1} \\
	&\sup_{0<t<T} \|u(\cdot,t)\|_{L^{q_c,\infty}(\R^n)} <\infty. 
	\label{eq:Lqcinfbdd}
\end{align}
For $p>p_{JL}:=(n-2\sqrt{n-1})/(n-4-2\sqrt{n-1})$ 
($p_L>p_{JL}>p_S$) with $n\geq11$, 
by \cite[Corollary 3]{MS21}, there exists 
a radial type~II blow-up solution 
with $0\leq u_0(x)\leq C(1+|x|)^{-\alpha}$  ($\alpha>2/(p-1)$). 
We see from \cite[Subsection 4.1]{MM09} and \cite[Proposition C.1]{MM11} 
that \eqref{eq:unif2p1} and \eqref{eq:Lqcinfbdd} are satisfied. 
In contrast, it was shown \cite[Proposition 7.1]{GK87} that  
if $p<p_S$, then 
all radially decreasing positive solutions do not satisfy either 
\eqref{eq:unif2p1} or \eqref{eq:Lqcinfbdd}, 
where the general situation in $p<p_S$ is still unclear as far as the authors know. 
\end{remark}

\begin{remark}[Blow-up profile in critical Lorentz space]
Our arguments in Sections \ref{sec:GKene}, \ref{sec:defect} and  \ref{sec:eps}
are also valid under the boundedness of the $L^{q_c,\infty}$ norm 
by slight modifications. 
Then, as a consequence of Lemma \ref{lem:noshrink} 
(the nondegeneracy of the blow-up limit), 
we obtain a nontrivial backward self-similar solution in the energy space. 

\begin{corollary}
Let $n\geq3$, $p>p_S$ and $u$ be a classical solution of \eqref{eq:main} 
with $u_0\in L^{q_c,\infty}(\R^n)$ and 
the finite maximal existence time $T>0$. 
Assume that $u$ has a blow-up point $x_0\in \R^n$ and satisfies 
\begin{equation}\label{eq:limLo}
	\liminf_{t\to T} \|u(\cdot,t)\|_{L^{q_c,\infty}(\R^n)}<\infty. 
\end{equation}
Then there exists a sequence $\lambda_k\to 0$ as $k\to \infty$ 
such that the rescaled solutions 
$\lambda_k^{2/(p-1)} u(x_0+\lambda_k x, T+\lambda_k^2 t)$ 
converge to $\overline{u}(x,t)\not \equiv 0$ strongly in 
$L^{p+1}_\loc(\R^n\times(-\infty,0))$ 
and $\dot{H}^1_\loc(\R^n\times(-\infty,0))$.  
Here $\overline{u}$ is a weak solution 
in $\R^n\times(-\infty,0)$ 
satisfying 
$\overline{u}(x,t)=\lambda^{2/(p-1)} \overline{u}(\lambda x, \lambda^2 t)$
for any $\lambda>0$. 
\end{corollary}

A related result of this corollary was obtained by Blatt-Struwe \cite[Theorem 6.9]{BS15}. 
Even if \eqref{eq:limLo} is not assumed, they proved that 
there exists $\lambda_k\to 0$ such that the rescaled solutions 
converge to a nontrivial backward self-similar solution
smoothly away from an exceptional set 
of locally finite $(n+2-2(p+1)/(p-1))$-dimensional 
parabolic Hausdorff measure. 
\end{remark}

\begin{remark}[Blow-up rate]
Specifying the blow-up rate of the critical norm 
for general blow-up solutions 
seems to be a fascinating and challenging open problem. 
We note that, for $p\geq p_{JL}$, 
there exist type~II blow-up solutions in $C([0,T); L^{q_c}(\R^n))$ 
with the logarithmic blow-up rate of the $L^{q_c}(\R^n)$ norm, 
see \cite[Corollary 1.5]{Se18} for $p=p_{JL}$
and \cite[Corollary 3]{MS21} for $p>p_{JL}$. 
See also \cite[Subsection 8.2.1]{Ha20s} for $p=p_S$ with $n=6$. 
\end{remark}

\begin{remark}[General domain]
It is very likely that Theorem \ref{th:main} holds 
with $\R^n$ replaced by a $C^{2+\alpha}$ domain for some $0<\alpha<1$. 
At least, it was proved \cite{MTpre} that 
$\limsup_{t\to T}\|u(\cdot,t)\|_{L^{q_c}(\Omega)} =\infty$ holds 
for $p>p_S$ under the $0$-Dirichlet boundary condition 
on an arbitrary $C^{2+\alpha}$ domain $\Omega$. 
We do not pursue this (technical) subject in this paper. 
\end{remark}

\subsection*{Organization of the paper}
In Section \ref{sec:GKene}, we define rescaled solutions $u_k$. 
By analyzing the Giga--Kohn energy, 
we give uniform estimates of $u_k$. 
In Section \ref{sec:defect}, 
we construct a blow-up limit and defect measures 
based on the uniform estimates. 
Then we examine the measures and 
prove strong convergence properties of $u_k$. 
In Section \ref{sec:eps}, 
we give an $\eps$-regularity theorem for the original solution, 
and then we also show an analogous result for the 
blow-up limit by applying the strong convergence. 
In Section \ref{sec:pr}, we prove Theorem \ref{th:main} 
based on the $\eps$-regularity results.

\section{Giga--Kohn energy}\label{sec:GKene}
In this section, we define rescaled solutions 
and give auxiliary estimates based on the analysis 
of the Giga--Kohn energy introduced in \cite{GK85}. 
We note that 
Lemmas \ref{lem:unienk}, \ref{lem:up1fixt} and 
\ref{lem:ukuni1} are the basis for constructing defect measures 
in Section \ref{sec:defect}. 
Lemmas \ref{lem:Ep1}, \ref{lem:wp1ene} and \ref{lem:wp1Ek} 
play a crucial role for proving the improved $\eps$-regularity results
in Section \ref{sec:eps}.

In what follows, we always assume $p>p_S$. 
After a scaling and a time shift for \eqref{eq:main}, 
we consider the solution of 
\begin{equation}\label{eq:fujita}
\left\{ 
\begin{aligned}
	&u_t-\Delta u=|u|^{p-1}u &&\mbox{ in }\R^n\times(-1,0), \\
	&u(\cdot,-1)=u_0\in  L^{q_c}(\R^n) &&\mbox{ in }\R^n, 
\end{aligned}
\right.
\end{equation} 
see Remark \ref{rem:defsol} for existence and uniqueness. 
Note that 
\begin{equation}\label{eq:uspace}
	u\in C([-1,0); L^{q_c}(\R^n)) \cap 
	L^\infty_\loc((-1,0);L^\infty(\R^n)) \cap
	C^{2,1}(\R^n\times(-1,0)). 
\end{equation}
To prove Theorem \ref{th:main} by contradiction, we suppose that 
\[
	\liminf_{t\to 0}\|u(\cdot,t)\|_{L^{q_c}(\R^n)}< \infty. 
\]
More precisely, we suppose the following: 
\begin{equation}\label{eq:seq}
\left\{ 
\begin{aligned}
	&\mbox{There exist a constant $M>0$ 
	and a monotone increasing} \\
	&\mbox{sequence $\{T_k\}_{k=1}^\infty \subset (-1,0)$ 
	satisfying $T_k\to 0$ as $k\to \infty$} \\
	&\mbox{and $\|u(\cdot,T_k)\|_{L^{q_c}(\R^n)} \leq M$ 
	for all $k\geq1$.} 
\end{aligned}
\right. 
\end{equation}
Define a sequence of rescaled solutions $u_k$ by 
\begin{equation}\label{eq:ukdef}
	u_k(x,t):= (-T_k)^\frac{1}{p-1} u( \sqrt{-T_k} x, (-T_k) t ), 
	\quad 
	x\in\R^n, T_k^{-1} \leq t< 0. 
\end{equation}
We note that $u_k$ satisfies 
\begin{equation}\label{eq:ukfujita}
	\partial_t u_k -\Delta u_k = |u_k|^{p-1} u_k 
	\quad \mbox{ in }\R^n\times (T_k^{-1},0). 
\end{equation}
Moreover, by the scaling invariance of the $L^{q_c}(\R^n)$ norm 
and \eqref{eq:seq}, we have  
\begin{equation}\label{eq:Tkbdd}
	\|u_k(\cdot,-1)\|_{L^{q_c}(\R^n)} 
	=\|u(\cdot,T_k)\|_{L^{q_c}(\R^n)} \leq M  
	\quad 
	\mbox{ for }k\geq1. 
\end{equation}

For $\tilde x\in\R^n$ and $-1<t<\tilde t\leq0$, 
the Giga--Kohn energy $E$ is defined by 
\begin{equation}\label{eq:oriEdef}
\begin{aligned}
	&E(t) = E_{(\tilde x,\tilde t)}(t) \\
	&:= (\tilde t-t)^{\frac{2}{p-1}+1} 
	\int_{\R^n} \left( \frac{|\nabla u|^2}{2} - \frac{|u|^{p+1}}{p+1} 
	+\frac{|u|^2}{2(p-1)(\tilde t-t)} \right) K_{(\tilde x,\tilde t)} dx, 
\end{aligned}
\end{equation}
where $K_{(\tilde x,\tilde {t})}(x,t):= 
(\tilde t-t)^{-n/2} e^{-|x-\tilde x|^2/(4(\tilde t-t))}$ 
is the backward heat kernel, 
the integrands are evaluated at $(x,t)$ and 
we often suppress the center of the scaling $(\tilde x,\tilde t)$ 
when no confusion can arise. 
We define the backward similarity variables $(\eta,\tau)$ by 
\[
	\eta:=\frac{x-\tilde x}{(\tilde t- t)^{1/2}}, \quad \tau:= -\log(\tilde t -t). 
\]
We also define the backward rescaled solution $w$ 
and the corresponding Giga--Kohn energy $\cE$ by 
\begin{align}
	&w(\eta,\tau)=w_{(\tilde x, \tilde t)}(\eta,\tau)
	:= e^{-\frac{1}{p-1}\tau} 
	u(\tilde x+e^{-\frac{1}{2}\tau}\eta,  \tilde t - e^{-\tau}), \notag\\
	&\cE(\tau)=\cE_{(\tilde x,\tilde t)}(\tau) 
	:= \int_{\R^n} \left( \frac{|\nabla w|^2}{2} 
	- \frac{|w|^{p+1}}{p+1} + \frac{|w(\eta,\tau)|^2}{2(p-1)} \right) 
	\rho(\eta) d\eta, 
	\label{eq:cEdef}
\end{align}
where $\rho(\eta):=e^{-|\eta|^2/4}$. 
We note that 
\begin{equation}\label{eq:EcE}
	E_{(\tilde x, \tilde t)}(t)
	=\cE_{(\tilde x, \tilde t)}(\tau) \quad \mbox{ with }\tau=-\log(\tilde t-t). 
\end{equation}
Moreover, by setting $\tau_0:=-\log(\tilde t+1)$, 
we see that $w$ satisfies 
\begin{equation}\label{eq:weq}
	\rho w_\tau 
	= \nabla \cdot (\rho \nabla  w) - \frac{1}{p-1} w\rho +|w|^{p-1}w\rho
	\quad \mbox{ in } \R^n\times(\tau_0, \infty).  
\end{equation}

Giga and Kohn \cite{GK85,GK87} proved the following monotonicity formula: 

\begin{lemma}\label{lem:Emono}
For any $\tilde x\in \R^n$, $-1<\tilde t\leq0$ and $\tau>\tau_0=-\log(\tilde t+1)$, 
\begin{equation}\label{eq:GKmn}
	\frac{d \cE_{(\tilde x,\tilde t)}}{d\tau}(\tau)
	= - \int_{\R^n} 
	|\partial_\tau w_{(\tilde x,\tilde t)}(\eta,\tau)|^2 
	\rho(\eta) d\eta. 
\end{equation}
In particular, $\cE_{(\tilde x,\tilde t)}(\tau)$ 
is nonincreasing in $\tau$. 
\end{lemma}

\begin{proof}
Integrating by parts gives 
\[
	\frac{d}{d\tau} \int_{\R^n} |\nabla w|^2 \rho d\eta
	= - \frac{d}{d\tau} \int_{\R^n} w \nabla\cdot (\rho\nabla w) d\eta
	= - 2\int_{\R^n} w_\tau  \nabla\cdot (\rho\nabla w) d\eta. 
\]
This together with 
differentiating the other terms in \eqref{eq:cEdef} shows \eqref{eq:GKmn}. 
See \cite[Proposition 2.1]{GK87} for details. 
\end{proof}

\begin{remark}\label{rem:wtSS}
By a direct computation, $w_\tau$ satisfies 
\[
\begin{aligned}
	&\partial_\tau w_{(\tilde x,\tilde t)}(\eta,\tau) = 
	- (\tilde t-t)^\frac{1}{p-1} S_{(\tilde x,\tilde t)}[u](x,t),  \\
	&S_{(\tilde x,\tilde t)}[u](x,t) := 
	\frac{u(x,t)}{p-1} + \frac{x-\tilde x}{2} \cdot \nabla u(x,t) 
	- (\tilde t-t ) u_t(x,t). 
\end{aligned}
\]
As stated in \cite[Page 3]{GK87}, 
$w_\tau$ measures the extent to which $u$ is not self-similar. 
This observation will play an important role 
to show the backward self-similarity of a blow-up limit, 
see Lemma \ref{lem:BSS}. 
\end{remark}

The monotonicity of $\cE$ guarantees the following properties: 

\begin{lemma}\label{lem:cEprop}
There exists $C>0$ depending only on $n$ and $p$ such that 
\begin{align}
	&\begin{aligned}
	&\frac{1}{2} \frac{d}{d\tau} \int_{\R^n} |w_{(\tilde x,\tilde t)}(\eta,\tau)|^2 
	\rho(\eta) d\eta \\
	&= -2\cE_{(\tilde x,\tilde t)}(\tau) 
	+ \frac{p-1}{p+1} \int_{\R^n} |w_{(\tilde x,\tilde t)}(\eta,\tau)|^{p+1} 
	\rho(\eta) d\eta, 
	\end{aligned}
	\label{eq:w22}  \\
	&\cE_{(\tilde x,\tilde t)}(\tau)\geq 0, \quad 
	\int_{\R^n} |w_{(\tilde x,\tilde t)}(\eta,\tau)|^2 \rho(\eta) d\eta 
	\leq C\cE_{(\tilde x,\tilde t)}^\frac{2}{p+1}(\tau'), 
	\label{eq:wno2}
\end{align}
for any $\tilde x\in \R^n$, $-1 < \tilde t\leq 0$ 
and $\tau>\tau'>\tau_0$. 
\end{lemma}

\begin{proof}
From \eqref{eq:weq} and the integration by parts, it follows that 
\[
	\frac{1}{2} \frac{d}{d\tau} \int_{\R^n} |w|^2 \rho d\eta
	= \int_{\R^n} \left( 
	- |\nabla  w|^2 - \frac{|w|^2}{p-1} +|w|^{p+1} 
	\right) \rho d\eta. 
\]
Then \eqref{eq:cEdef} shows \eqref{eq:w22}. 
Lemma \ref{lem:Emono} and the H\"older inequality yield 
\[
	\frac{1}{2} \frac{d}{d\tau} \int_{\R^n} |w(\eta,\tau)|^2 \rho(\eta) d\eta
	\geq  -2\cE(\tau') 
	+ C \left( \int_{\R^n} |w(\eta,\tau)|^2 \rho(\eta) d\eta \right)^\frac{p+1}{2} 
\]
for all $\tau_0<\tau'<\tau<\infty$. 
Since the length of the time-interval for $\tau$ is infinite, 
this ordinary differential inequality guarantees \eqref{eq:wno2}. 
For details, 
see \cite[Propositions 2.1, 2.2]{GK87} 
and \cite[Proposition 23.8]{QSbook2}. 
\end{proof}

By $u\in C([-1,0); L^{q_c}(\R^n))$ as in \eqref{eq:uspace}, 
we can take a constant $\delta_1>0$ depending on $T_1$ and 
satisfying $(1+\delta_1)T_1>-1$ 
such that 
\begin{equation}\label{eq:cM}
	\|u(\cdot,t)\|_{L^{q_c}(\R^n)} \leq 2M 
	\quad \mbox{ for }(1+\delta_1)T_1 \leq t\leq T_1. 
\end{equation}
The monotonicity together with \eqref{eq:seq} 
shows a uniform estimate of $\cE$.

\begin{lemma}\label{lem:unibE}
There is $C>0$ depending only on 
$n$, $p$, $M$ and $\delta_1$ such that 
\[
	\cE_{(\tilde x,\tilde t)}(\tau) \leq C
\]
for all $\tilde x\in\R^n$ and $T_1< \tilde t\leq 0$ 
and $\tau\geq  \tau_1:=-\log(\tilde t-T_1)$. 
\end{lemma}

\begin{proof}
From the monotonicity of $\cE$ and the switch back to the original variables, 
it follows that 
\[
\begin{aligned}
	\cE(\tau) &\leq \cE(\tau_1) \leq 
	C \int_{\R^n} ( |\nabla w(\eta,\tau_1)|^2  
	+ |w(\eta,\tau_1)|^2 ) \rho(\eta) d\eta \\
	&\leq 
	C (\tilde t-T_1)^{\frac{2}{p-1}+1} 
	\int_{\R^n} \left( 
	|\nabla u(x,T_1)|^2 + \frac{|u(x,T_1)|^2}{\tilde t-T_1} \right)
	K_{(\tilde x,\tilde t)}(x,T_1) dx. 
\end{aligned}
\]
Remark that $q_c>2$ by $p>p_S$. 
Then the H\"older inequality and \eqref{eq:seq} yield 
\[
\begin{aligned}
	\int_{\R^n} |u(x,T_1)|^2 K_{(\tilde x,\tilde t)}(x,T_1) dx
	&\leq 
	M^2 (\tilde t-T_1)^{-\frac{n}{2}} 
	\left( \int_{\R^n} e^{-\frac{|x-\tilde x|^2}{C(\tilde t-T_1)}} dx 
	\right)^{1-\frac{2}{q_c}} \\
	&\leq CM^2 (\tilde t-T_1)^{-\frac{2}{p-1}}. 
\end{aligned}
\]

Since $u$ is a solution of \eqref{eq:fujita}, 
$u$ satisfies the following integral equation: 
\[
\begin{aligned}
	u(x,T_1) &= 
	\int_{\R^n} G(x-y,- \delta_1 T_1) u(y,(1+\delta_1)T_1) dy \\
	&\quad 
	+ \int_{(1+\delta_1)T_1}^{T_1} \int_{\R^n} 
	G(x-y,T_1-s) |u(y,s)|^{p-1} u(y,s) dyds, 
\end{aligned}
\]
where $G(x,t):= (4\pi t)^{-n/2} e^{-|x|^2/(4t)}$ is the heat kernel. 
Setting $K_1(x,t):=t^{-(n+1)/2}e^{-|x|^2/(8t)}$ 
gives  $|\nabla G| \leq C|x|t^{-1}G\leq C K_1$, and so 
\begin{equation}\label{eq:nabuT1}
\begin{aligned}
	|\nabla u(x,T_1)| &\leq 
	C \int_{\R^n} K_1(x-y,- \delta_1 T_1) |u(y,(1+\delta_1)T_1)| dy  \\
	&\quad 
	+ C \int_{(1+\delta_1)T_1}^{T_1} \int_{\R^n} 
	K_1(x-y,T_1-s) |u(y,s)|^p dyds \\
	&=: C U_1(x) + C U_2(x). 
\end{aligned}
\end{equation}

By direct computations, we see that 
\[
\begin{aligned}
	&\int_{\R^n} |\nabla u(x,T_1)|^2 K_{(\tilde x,\tilde t)}(x,T_1) dx 
	\leq 
	C\int_{\R^n} (|U_1|^2+|U_2|^2) K_{(\tilde x,\tilde t)}(x,T_1) dx 
	\\
	&\leq 
	C\| U_1\|_{L^\infty(\R^n)}^2
	 \int_{\R^n} K_{(\tilde x,\tilde t)} dx
	+ 
	C (\tilde t-T_1)^{-\frac{n}{2}} 
	\int_{\R^n}\left(  |U_2(x)| 
	e^{-\frac{|x-\tilde x|^2}{8(\tilde t-T_1)}} \right)^2 dx 
	\\
	&= C(4\pi)^\frac{n}{2} \| U_1\|_{L^\infty(\R^n)}^2  
	+ C (\tilde t-T_1)^{-\frac{n}{2}} 
	\left\| U_2 e^{-\frac{|\cdot-\tilde x|^2}{8(\tilde t-T_1)}}
	\right\|_{L^{2,2}(\R^n)}^2. 
\end{aligned}
\]
From the convolution inequality and 
the H\"older inequality in the Lorentz spaces (see \cite[Proposition 2.1]{KY99}) 
with the aid of $n(p-1)/(p+1)>2$ by $p>p_S$, it follows that 
\begin{equation}\label{eq:nabuT1T1}
\begin{aligned}
	&\int_{\R^n} |\nabla u(x,T_1)|^2 K_{(\tilde x,\tilde t)}(x,T_1) dx \\
	&\leq 
	C ( (-\delta_1 T_1)^{-\frac{n}{2q_c}-\frac{1}{2}} 
	\|u(\cdot,(1+\delta_1)T_1)\|_{L^{q_c}(\R^n)} )^2  \\
	&\quad 
	+ C (\tilde t-T_1)^{-\frac{n}{2}} \| U_2 \|_{L^{\frac{n(p-1)}{p+1},\infty}(\R^n)}^2  
	\left\| e^{-\frac{|\cdot-\tilde x|^2}{8(\tilde t-T_1)}} 
	\right\|_{L^{\frac{2n(p-1)}{n(p-1)-2(p+1)},2} (\R^n)}^2. 
\end{aligned}
\end{equation}
We note that $-\delta_1 T_1\geq \delta_1 (\tilde t-T_1)>0$. 
By the same argument as in \cite[Proposition 2.1]{MTpre} 
(see also \cite[Theorem 18.1]{Me97}),  
we obtain 
\[
	\| U_2\|_{L^{\frac{n(p-1)}{p+1},\infty}(\R^n)} \leq 
	C \sup_{(1+\delta_1)T_1<t<T_1}\| u(\cdot,t)\|_{L^{q_c,\infty}(\R^n)}^p 
	\leq CM^{2p}. 
\]
This together with \eqref{eq:cM} 
and the straightforward computations shows that 
\[
	\int_{\R^n} |\nabla u(x,T_1)|^2 K_{(\tilde x,\tilde t)} (x,T_1) dx 
	\leq 
	C (\tilde t-T_1)^{-\frac{2}{p-1}-1} 
	( \delta_1^{-\frac{2}{p-1}-1} M^2 + M^{2p} ). 
\]
Then the lemma follows. 
\end{proof}

As an application of Lemmas \ref{lem:cEprop} and \ref{lem:unibE}, 
we prove that time integrals of $\cE$ can be controlled by 
space-time integrals of $|w|^{p+1}$.

\begin{lemma}\label{lem:Ep1}
There is $C>0$ depending only on 
$n$, $p$, $M$ and $\delta_1$ such that 
\[
\begin{aligned}
	\int_{\tau'}^\tau \cE_{(\tilde x,\tilde t)}(\sigma) d\sigma 
	& \leq 
	C \int_{\tau'}^\tau \int_{\R^n} 
	|w_{(\tilde x,\tilde t)}|^{p+1} \rho d\eta d\sigma \\
	&\quad 
	+ C (\tau-\tau')^\frac{p-1}{2(p+1)}
	\left(\int_{\tau'}^\tau \int_{\R^n} 
	|w_{(\tilde x,\tilde t)}|^{p+1} \rho d\eta d\sigma 
	\right)^\frac{1}{p+1} 
\end{aligned}
\]
for all $\tilde x\in \R^n$, $T_1<\tilde t\leq0$ 
and $\tau>\tau'>\tau_1=-\log(\tilde t-T_1)$. 
\end{lemma}

\begin{proof}
From \eqref{eq:w22}, the H\"older inequality and 
\eqref{eq:GKmn}, it follows that 
\begin{equation}\label{eq:etp1}
\begin{aligned}
	\cE(\tau) 
	&\leq 
	C\left(\int_{\R^n} |w_\tau|^2\rho d\eta \right)^\frac{1}{2} 
	\left(\int_{\R^n} |w|^2\rho d\eta \right)^\frac{1}{2}
	+ C \int_{\R^n} |w|^{p+1} \rho d\eta \\
	&= C\left( -\frac{d\cE}{d\tau} \right)^\frac{1}{2} 
	\left(\int_{\R^n} |w|^2\rho d\eta \right)^\frac{1}{2}
	+ C \int_{\R^n} |w|^{p+1} \rho d\eta. 
\end{aligned}
\end{equation}
By the H\"older inequality, Lemma \ref{lem:unibE} and $\cE\geq0$, 
we see that 
\[
\begin{aligned}
	\int_{\tau'}^\tau \cE(\sigma) d\sigma 
	&\leq 
	C( \cE(\tau')-\cE(\tau) )^\frac{1}{2}  
	\left( \iint |w|^2\rho  \right)^\frac{1}{2} 
	+ C \iint |w|^{p+1} \rho  \\
	&\leq 
	C \left( \iint |w|^{p+1} \rho 
	\right)^\frac{1}{p+1} 
	\left( \iint \rho \right)^\frac{p-1}{2(p+1)} + C \iint |w|^{p+1} \rho, 
\end{aligned}
\]
where we write 
$\iint(\cdots)=\int_{\tau'}^\tau \int_{\R^n} (\cdots) d\eta d\sigma$, 
and the lemma follows. 
\end{proof}

The same argument as in Lemma \ref{lem:Ep1} with the monotonicity of $\cE$ 
also shows a kind of reverse relation, that is, 
space-time integrals of $|w|^{p+1}$ can be controlled by $\cE$.

\begin{lemma}\label{lem:wp1ene}
There is $C>0$ depending only on 
$n$, $p$, $M$ and $\delta_1$ such that 
\[
	\int_{\tau'}^\tau \int_{\R^n} 
	|w_{(\tilde x,\tilde t)}(\eta,\sigma)|^{p+1} \rho(\eta) d\eta d\sigma 
	\leq C f( \cE_{(\tilde x,\tilde t)}(\tau') )
	f(\tau-\tau') 
\]
for all $\tilde x\in \R^n$, $T_1<\tilde t\leq0$ 
and $\tau>\tau'>\tau_1$, where $f(s):=s+s^{1/2}$ for $s\geq0$. 
\end{lemma}

\begin{proof}
By the same computations as in \eqref{eq:etp1}, 
we have 
\begin{equation}\label{eq:wp1w2E}
\begin{aligned}
	\int_{\R^n} |w|^{p+1} \rho d\eta 
	&= 
	\frac{p+1}{2(p-1)} \frac{d}{d\tau} \int_{\R^n} w^2 \rho d\eta 
	+ \frac{2(p+1)}{p-1} \cE(\tau) \\
	&\leq  C\left( -\frac{d\cE}{d\tau} \right)^\frac{1}{2} 
	\left(\int_{\R^n} w^2\rho d\eta \right)^\frac{1}{2}
	+ C \cE(\tau). 
\end{aligned}
\end{equation}
Then by \eqref{eq:wno2} and the monotonicity of $\cE$, we see that 
\[
\begin{aligned}
	\iint |w|^{p+1} \rho 
	&\leq 
	C( \cE(\tau')-\cE(\tau) )^\frac{1}{2}  
	\left( \iint w^2\rho \right)^\frac{1}{2} 
	+ C \int_{\tau'}^\tau \cE(\sigma) d\sigma \\
	&\leq 
	C\cE(\tau')^\frac{1}{2}  
	( (\tau-\tau') \cE^\frac{2}{p+1}(\tau') )^\frac{1}{2} 
	+ C \cE(\tau') (\tau-\tau'), 
\end{aligned}
\]
where 
$\iint(\cdots)=\int_{\tau'}^\tau \int_{\R^n} (\cdots) d\eta d\sigma$. 
This together with 
$s^{1/2+1/(p+1)}\leq Cf(s)$ for $s\geq0$ shows the lemma. 
\end{proof}

By Lemma \ref{lem:wp1ene}, 
we give uniform estimates of Morrey type. 
Throughout this paper, we write 
$B_r(x):= \{y\in \R^n; |x-y|<r\}$ and $B_r:=B_r(0)$ 
for $r>0$ and $x\in\R^n$.

\begin{lemma}\label{lem:unien}
There exists $C>0$ depending only on $n$, $p$, 
$M$ and $\delta_1$ such that  
the following (i) and (ii) hold for any $\tilde x\in\R^n$ and $T_1<t<0$. 
\begin{itemize}
\item[(i)]
$\displaystyle (-t)^{\frac{2}{p-1}-\frac{n}{2}} \int_t^{t/2} 
\int_{B_{\sqrt{-t/2}}(\tilde x)} 
( |\nabla u(x,s)|^2 + |u(x,s)|^{p+1} ) dxds 
\leq C$. 
\item[(ii)]
$\displaystyle (-t)^{\frac{2}{p-1}-\frac{n}{2}+1} \int_t^{t/2} 
\int_{B_{\sqrt{-t/2}}(\tilde x)} 
u_t^2(x,s) dxds 
\leq C$. 
\end{itemize}
\end{lemma}

\begin{proof}
Let $\tilde x\in\R^n$, $T_1< t< 0$, 
$\tau=-\log(-t/2)$ and $\tau'=-\log(-t)$. 

(i) By Lemma \ref{lem:wp1ene} with Lemma \ref{lem:unibE} 
and $\tau-\tau'=\log 2$, we obtain 
\[
\begin{aligned}
	&(-t)^{\frac{2}{p-1}-\frac{n}{2}} \int_t^{t/2} 
	\int_{B_{\sqrt{-t/2}}(\tilde x)} |u(x,s)|^{p+1}  dxds  \\
	&\leq 
	C\int_t^{t/2} (-s)^\frac{2}{p-1} \int_{\R^n}
	|u|^{p+1} K_{(\tilde x,0)}dxds 
	\leq 
	C\int_{\tau'}^\tau \int_{\R^n}
	|w_{(\tilde x,0)}|^{p+1} \rho d\eta d\sigma \leq C. 
\end{aligned}
\]
This together with  \eqref{eq:cEdef} and Lemma \ref{lem:unibE} 
shows the estimate for $|\nabla u|^2$.

(ii) Lemmas \ref{lem:Emono} and \ref{lem:unibE} show that  
$\int_{\tau'}^\tau \int_{\R^n} 
|w_\tau|^2 \rho d\eta d\sigma = \cE(\tau') - \cE(\tau) \leq C$, 
where we note that 
$\partial_\tau w_{(\tilde x, 0)}(\eta,\tau)
= -(-t)^{1/(p-1)} S_{(\tilde x,0)}[u](x,t)$ 
and $S$ is given in Remark \ref{rem:wtSS}. 
Switching back to the original variables, we have 
\[
	\int_t^{t/2} (-s)^{\frac{2}{p-1}-\frac{n}{2}-1}
	\int_{\R^n} | S_{(\tilde x,0)}[u](x,s) |^2 
	e^{-\frac{|x-\tilde x|^2}{4(-s)}} dxds 
	\leq C. 
\]
Then, $|x-\tilde x|^2 e^{-|x-\tilde x|^2/(4(-s))} \leq C(-s)$  gives 
\[
\begin{aligned}
	C &\geq 
	(-t)^{\frac{2}{p-1}-\frac{n}{2}-1}
	\int_t^{t/2} \int_{B_{\sqrt{-t/2}}(\tilde x)} 
	| S_{(\tilde x,0)}[u](x,s) |^2 
	e^{-\frac{|x-\tilde x|^2}{4(-s)}} dxds \\
	&\geq 
	(-t)^{\frac{2}{p-1}-\frac{n}{2}-1} 
	\int_t^{t/2} \int_{B_{\sqrt{-t/2}}(\tilde x)} 
	\left( \frac{(-s)^2 |u_t|^2}{C} 
	- C|u|^2 - C(-t) |\nabla u|^2  \right) dxds
\end{aligned}
\]
Hence by (i) and the H\"older inequality, we obtain the desired estimate. 
\end{proof}

In what follows, we give estimates of $u_k$ defined by \eqref{eq:ukdef}. 
For $\tilde x\in\R^n$, $k\geq1$ and $T_k^{-1}<t<\tilde t\leq0$, 
define  
\begin{equation}\label{eq:EkwkcEk}
\left\{ 
\begin{aligned}
	&E_k(t)=E_{k,(\tilde x,\tilde t)}(t) := (\tilde t-t)^{\frac{2}{p-1}+1} \\
	&\quad \times \int_{\R^n} \left( \frac{|\nabla u_k|^2}{2} - \frac{|u_k|^{p+1}}{p+1} 
	+\frac{|u_k(x,t)|^2}{2(p-1)(\tilde t-t)} \right) K_{(\tilde x,\tilde t)} dx, \\
	&w_k(\eta,\tau)=w_{k,(\tilde x, \tilde t)}(\eta,\tau)
	:= e^{-\frac{1}{p-1}\tau} 
	u_k(\tilde x+e^{-\frac{1}{2}\tau}\eta,  \tilde t - e^{-\tau}), \\
	&\cE_k(\tau)=\cE_{k,(\tilde x,\tilde t)}(\tau) 
	:= \int_{\R^n} \left( \frac{|\nabla w_k|^2}{2} 
	- \frac{|w_k|^{p+1}}{p+1} + \frac{|w_k|^2}{2(p-1)} \right) \rho d\eta. 
\end{aligned}
\right. 
\end{equation}
where $\rho(\eta)=e^{-|\eta|^2/4}$. 
We note that 
\begin{equation}\label{eq:EcEk}
	E_{k,(\tilde x,\tilde t)}(t)
	=\cE_{k,(\tilde x,\tilde t)}(\tau) \quad \mbox{ with }\tau=-\log(\tilde t-t). 
\end{equation}
Since $u_k$ satisfies \eqref{eq:ukfujita}, 
the backward rescaled sequence $w_k$ satisfies 
\[
	\rho  \partial_\tau w_k 
	= \nabla \cdot (\rho \nabla  w_k) - \frac{1}{p-1} w_k\rho +|w_k|^{p-1}w_k\rho
	\quad \mbox{ in } \R^n\times(\tau_k, \infty), 
\]
where $\tau_k:=-\log(\tilde t - T_k^{-1})$. 
Then the following monotonicity formula holds:

\begin{lemma}\label{lem:monok}
For any $\tilde x\in \R^n$, $k\geq1$, $T_k^{-1}<\tilde t\leq0$ 
and $\tau>\tau_k$, 
\begin{equation}\label{eq:Emonok}
	\frac{d \cE_{k,(\tilde x,\tilde t)}}{d\tau}(\tau)
	= - \int_{\R^n} 
	|\partial_\tau w_{k,(\tilde x,\tilde t)}(\eta,\tau)|^2 
	\rho(\eta) d\eta. 
\end{equation}
In particular, $\cE_{(\tilde x,\tilde t)}(\tau)$ 
is nonincreasing in $\tau$. 
Moreover, there exists $C>0$ depending only on $n$ and $p$ such that 
\begin{align}
	&\begin{aligned}
	&\frac{1}{2} \frac{d}{d\tau} \int_{\R^n} |w_{k,(\tilde x,\tilde t)}(\eta,\tau)|^2 
	\rho(\eta) d\eta \\
	&= -2\cE_{k,(\tilde x,\tilde t)}(\tau) 
	+ \frac{p-1}{p+1} \int_{\R^n} |w_{k,(\tilde x,\tilde t)}(\eta,\tau)|^{p+1} 
	\rho(\eta) d\eta,  
	\end{aligned}
	\label{eq:w22k} \\
	&\cE_{k,(\tilde x,\tilde t)}(\tau)\geq 0, \quad 
	\int_{\R^n} |w_{k,(\tilde x,\tilde t)}(\eta,\tau)|^2 \rho(\eta) d\eta 
	\leq C\cE_{k,(\tilde x,\tilde t)}^\frac{2}{p+1}(\tau'),  \label{eq:wno2k}
\end{align}
for any $k\geq1$ and $\tau>\tau'>\tau_k$. 
\end{lemma}

\begin{proof}
Since $w_k$ satisfies the same equation as that of $w$, 
this lemma follows from the same computations as in 
Lemmas \ref{lem:Emono} and \ref{lem:cEprop}. 
\end{proof}

We give a uniform estimate for $\cE_k$. 

\begin{lemma}\label{lem:unibEk}
There is $C>0$ depending only on 
$n$, $p$, $M$ and $\delta_1$ such that 
\[
	\cE_{k,(\tilde x,\tilde t)}(\tau) \leq C 
\]
for all $\tilde x\in\R^n$, $k\geq1$, $-T_1/T_k< \tilde t\leq 0$ 
and $\tau\geq \tau_k':=-\log(\tilde t+T_1/T_k)$. 
\end{lemma}

\begin{proof}
From \eqref{eq:EcEk}, the change of variables and \eqref{eq:EcE}, 
it follows that 
\begin{equation}\label{eq:EkEtra}
\begin{aligned}
	\cE_{k,(\tilde x,\tilde t)}(\tau)
	&= E_{k,(\tilde x,\tilde t)}(t)
	= E_{(\sqrt{-T_k}\tilde x, -T_k \tilde t)}(-T_k t) \\
	&= \cE_{(\sqrt{-T_k}\tilde x, -T_k \tilde t)}(\tau-\log(-T_k)). 
\end{aligned}
\end{equation}
This together with Lemma \ref{lem:unibE} implies that 
$\cE_{k,(\tilde x,\tilde t)}(\tau) \leq C$ 
if $T_1<-T_k \tilde t\leq0$ and 
$\tau-\log(-T_k) \geq -\log( -T_k \tilde t -T_1 )$. 
Then the lemma follows. 
\end{proof}

Lemmas \ref{lem:monok} and \ref{lem:unibEk} 
deduce analogs of Lemmas \ref{lem:Ep1} and \ref{lem:wp1ene}. 

\begin{lemma}\label{lem:wp1Ek}
There is $C>0$ depending only on 
$n$, $p$, $M$ and $\delta_1$ such that 
\begin{align}
	&\begin{aligned}
	\int_{\tau'}^\tau \cE_{k,(\tilde x,\tilde t)} d\sigma 
	&\leq 
	C \int_{\tau'}^\tau \int_{\R^n} |w_{k,(\tilde x,\tilde t)}|^{p+1} 
	\rho d\eta d\sigma  \\
	&\quad 
	+ C (\tau-\tau')^\frac{p-1}{2(p+1)}
	\left(\int_{\tau'}^\tau \int_{\R^n} |w_{k,(\tilde x,\tilde t)}|^{p+1} 
	\rho d\eta d\sigma 
	\right)^\frac{1}{p+1}, 
	\end{aligned}
	\label{eq:Ep1k}   \\
	&\int_{\tau'}^\tau \int_{\R^n} 
	|w_{k,(\tilde x,\tilde t)}(\eta,\sigma)|^{p+1} \rho(\eta) d\eta d\sigma 
	\leq C f( \cE_{k,(\tilde x,\tilde t)}(\tau') )
	f(\tau-\tau'),   \label{eq:p1Ek}
\end{align}
for all $\tilde x\in\R^n$, $k\geq1$, $-T_1/T_k< \tilde t\leq 0$ 
and $\tau\geq \tau_k'$. 
\end{lemma}

\begin{proof}
Since \eqref{eq:w22k} holds, 
by the same computations as in \eqref{eq:etp1}, we have 
\[
	\cE_k(\tau) 
	\leq C\left( -\frac{d\cE_k}{d\tau} \right)^\frac{1}{2} 
	\left(\int_{\R^n} |w_k|^2\rho d\eta \right)^\frac{1}{2} 
	+ C \int_{\R^n} |w_k|^{p+1} \rho d\eta
\]
for $\tau>\tau_k$. 
By integrating this inequality with the aid of 
$\tau_k'>\tau_k$ for $k\geq1$ and Lemmas \ref{lem:monok} and \ref{lem:unibEk}, 
we obtain \eqref{eq:Ep1k} for $\tau\geq \tau_k'$. 
Similar computations to \eqref{eq:wp1w2E} show that 
\begin{equation}\label{eq:wp1w2Ek}
	\int_{\R^n} |w_k|^{p+1} \rho d\eta  
	\leq  C\left( -\frac{d\cE_k}{d\tau} \right)^\frac{1}{2} 
	\left(\int_{\R^n} |w_k|^2\rho d\eta \right)^\frac{1}{2} 
	+ C \cE_k(\tau) 
\end{equation}
for $\tau>\tau_k$. 
Integrating this inequality gives \eqref{eq:p1Ek}. 
\end{proof}

In the rest of this subsection, we give uniform estimates. 
By Lemma \ref{lem:unien}, we obtain estimates for space-time integrals of $u_k$.

\begin{lemma}\label{lem:unienk}
There is $C>0$ depending only on 
$n$, $p$, $M$ and $\delta_1$ such that 
\[
\begin{aligned}
	&\int_{-2R^2}^{-R^2} \int_{ B_R(\tilde x) } 
	(|\nabla u_k|^2 + |u_k|^{p+1}) dxdt 
	\leq C R^{n-\frac{4}{p-1}}, \\ 
	&\int_{-2R^2}^{-R^2} \int_{ B_R(\tilde x) }  |\partial_t u_k|^2  dxdt 
	\leq C R^{n-\frac{4}{p-1}-2}, 
\end{aligned}
\]
for any $\tilde x\in\R^n$, $R>0$ and 
$k$ satisfying $T_1/(2R^2)<T_k<0$. 
\end{lemma}

\begin{proof}
From the change of variables and Lemma \ref{lem:unien}, it follows that 
\[
\begin{aligned}
	&\int_{-2R^2}^{-R^2} \int_{ B_R(\tilde x) } 
	(|\nabla u_k|^2 + |u_k|^{p+1} ) dxdt \\
	&=(-T_k)^{\frac{2}{p-1}-\frac{n}{2}} 
	\int_{2T_k R^2}^{T_k R^2} 
	\int_{B_{\sqrt{-T_k} R}(\sqrt{-T_k} \tilde x)} 
	(|\nabla u|^2 + |u|^{p+1} ) dyds 
	\leq CR^{n-\frac{4}{p-1}}
\end{aligned}
\]
and that 
\[
\begin{aligned}
	&\int_{-2R^2}^{-R^2} \int_{ B_R(\tilde x) } 
	|\partial_t u_k|^2  dxdt  \\
	&=(-T_k)^{\frac{2}{p-1}-\frac{n}{2}+1} 
	\int_{2T_k R^2}^{T_k R^2} 
	\int_{B_{\sqrt{-T_k} R}(\sqrt{-T_k} \tilde x)} 
	|\partial_t u|^2  dyds 
	\leq C R^{n-\frac{4}{p-1}-2}. 
\end{aligned}
\]
Then the lemma follows. 
\end{proof}

Based on \eqref{eq:wp1w2Ek}, we prove a uniform estimate 
concerning space integrals of $u_k(\cdot,t)$ for $t$ fixed. 

\begin{lemma}\label{lem:up1fixt}
Let $R>0$. Then there exists a subsequence still denoted by $u_k$ 
such that the following holds: 
For a.e.~$t\in (-\infty,0)$, 
there exist constants $k_t\geq1$ depending on $t$ and 
$C_t>0$ depending only on 
$n$, $p$, $M$, $\delta_1$, $R$ and $t$ such that, for any $k\geq k_t$, 
\[
	\int_{B_R} ( |\nabla u_k(x,t)|^2 + |u_k(x,t)|^{p+1} 
	+ |\partial_t u_k(x,t)|^2 ) dx \leq C_t. 
\]
\end{lemma}

\begin{proof}
Throughout this proof, $C$ is independent of $t$ and $\tau$, 
but $C_t$ depends on $t$. 
These constants are independent of $k$. 
In addition, we take $(\tilde x,\tilde t)=(0,0)$ as the the center of the scaling. 
We note that $\tau_k=-\log(-T_k^{-1})$, 
$\tau_k'=-\log(T_1/T_k)$, $\tau_k'>\tau_k$ for $k\geq1$ 
and $\tau_k, \tau_k'\to -\infty$ as $k\to \infty$. 

We first estimate the integral of $|u_k(\cdot,t)|^{p+1}$. 
From \eqref{eq:wp1w2Ek}, the second inequality in \eqref{eq:wno2k} 
and Lemma \ref{lem:unibEk}, 
it follows that 
\begin{equation}\label{eq:wkp1t}
\begin{aligned}
	\int_{\R^n} |w_k|^{p+1} \rho d\eta  
	&\leq  C\left( -\frac{d\cE_k}{d\tau}(\tau) \right)^\frac{1}{2} 
	( \cE_k^\frac{2}{p+1}(\tau') )^\frac{1}{2}
	+ C \cE_k(\tau) \\
	&\leq  C\left( -\frac{d\cE_k}{d\tau}(\tau) \right)^\frac{1}{2} + C 
\end{aligned}
\end{equation}
for $\tau>\tau'\geq \tau_k'$. 
We claim that there exists a subsequence still denoted by 
$d\cE_k/d\tau$ such that 
\begin{equation}\label{eq:dEkto0}
	\mbox{for a.e.~}\tau\in\R, \quad 
	-\frac{d\cE_k}{d\tau}(\tau) \to 0  
	 \mbox{ as } k\to \infty. 
\end{equation}
This will be proved once we confirm that 
\begin{equation}\label{eq:Fkrto0}
	F_k(r_0):= \int_{3r_0/4}^{r_0} \frac{2}{s} 
	\int_{-\log(4s^2)}^{-\log (s^2)} 
	\left( -\frac{d\cE_k}{d\sigma}(\sigma) \right) d\sigma ds \to 0
\end{equation}
for each $r_0>0$ as $k\to\infty$. 
Indeed, since 
\begin{equation}\label{eq:Fkr0cal}
\begin{aligned}
	F_k(r_0) &\geq 
	\frac{2}{r_0} \int_{3r_0/4}^{r_0} 
	\int_{-\log(9r_0^2/4)}^{-\log (r_0^2)} 
	\left( -\frac{d\cE_k}{d\sigma}(\sigma) \right) d\sigma ds \\
	&\geq 
	\frac{1}{2} 
	\int_{-\log(9r_0^2/4)}^{-\log (r_0^2)} 
	\left( -\frac{d\cE_k}{d\sigma}(\sigma) \right) d\sigma\geq 0, 
\end{aligned}
\end{equation}
the validity of \eqref{eq:Fkrto0} implies that 
the sequence of nonnegative functions 
$-d\cE_k/d\sigma$ converges to $0$ strongly 
in $L^1(-\log((9/4)r_0^2), -\log (r_0^2))$ as $k\to\infty$. 
Thus, there exists a subsequence $-d\cE_k/d\tau$ 
such that $-d\cE_k/d\tau\to 0$ a.e.~in 
$(-\log((9/4)r_0^2), -\log (r_0^2))$. 
Since $r_0>0$ is arbitrary, \eqref{eq:dEkto0} follows. 

We prove \eqref{eq:Fkrto0}. Let $r_0>0$. 
By \eqref{eq:EcEk} and
the second equality in \eqref{eq:EkEtra}, we see that 
\begin{equation}\label{eq:Fkr0def}
\begin{aligned}
	F_k(r_0) &= 
	\int_{3r_0/4}^{r_0} \frac{2}{s} 
	( \cE_k(-\log (4s^2))-\cE_k(-\log (s^2)) ) ds  \\
	&= \int_{3r_0/4}^{r_0} \frac{2}{s} 
	( E(4T_ks^2) - E(T_ks^2) )  ds \\
	&= \int_{4T_kr_0^2}^{9T_kr_0^2/4} E(\theta) \frac{d\theta}{\theta} 
	- \int_{T_kr_0^2}^{9T_kr_0^2/16} E(\theta) \frac{d\theta}{\theta}. 
\end{aligned}
\end{equation}
For $0<\tilde r<1$, we define a function $H(\tilde r)$ by
\[
	H(\tilde r):= \int_{-\tilde r^2}^{-9\tilde r^2/16} 
	E(\theta) \frac{d\theta}{-\theta} 
	= \int_{-1}^{-9/16} 
	E(\tilde r^2 \theta) \frac{d\theta}{-\theta}. 
\]
Since $E$ is monotone decreasing and $\theta<0$, 
$H(\tilde r)\geq0$ is decreasing when $\tilde r$ is decreasing. 
Hence $\lim_{\tilde r\to0} H(\tilde r)$ exists. 
This together with \eqref{eq:Fkr0def} 
and $T_k\to0$ as $k\to\infty$ 
guarantees \eqref{eq:Fkrto0}. 
As stated above, \eqref{eq:dEkto0} also holds.

The estimate \eqref{eq:wkp1t} and the limit \eqref{eq:dEkto0} 
show that there exists a subsequence $w_k$ satisfying the following: 
For a.e.~$\tau\in \R$, there exist constants $k_\tau\geq1$ and 
$C>0$ such that 
\begin{equation}\label{eq:wktfix}
	\int_{\R^n} |w_k(\eta,\tau)|^{p+1} \rho(\eta) d\eta 
	\leq  C
\end{equation}
for $k\geq k_\tau$. 
By switching back to the variables of $u_k$, 
for a.e.~$t\in (-\infty,0)$, there exist constants $k_t\geq1$ and 
$C>0$ such that 
\[
	(-t)^{\frac{2}{p-1}+1} 
	\int_{\R^n} |u_k(x,t)|^{p+1} K_{(0,0)}(x,t) dx 
	\leq  C
\]
for $k\geq k_t$. Let $R>0$. Then, for a.e.~$t\in(-\infty,0)$, 
\[
	\int_{B_R} |u_k(x,t)|^{p+1}  dx \leq 
	C (-t)^{\frac{n}{2}-\frac{2}{p-1}-1} e^\frac{R^2}{4(-t)}
	\leq C_t. 
\]
for $k\geq k_t$. Hence we obtain the desired estimate 
for $|u_k(\cdot,t)|^{p+1}$.

Let us next estimate $|\nabla u_k(\cdot,t)|^2$. 
From the definition of $\cE_k$ in \eqref{eq:EkwkcEk}, 
\eqref{eq:wktfix} and 
Lemma \ref{lem:unibEk} with $\tau_k'\to-\infty$ as $k\to\infty$, 
it follows that, for a.e.~$\tau\in\R$, 
\[
	\int_{\R^n} |\nabla w_k|^2 \rho d\eta 
	\leq  \frac{2}{p+1} \int_{\R^n} |w_k|^{p+1} \rho d\eta
	+ 2\cE_k(\tau) \leq C
\]
for $k\geq k_\tau$. 
This implies the desired estimate for $|\nabla u_k(\cdot,t)|^2$.

Finally, we consider $|\partial_t u_k(\cdot,t)|^2$. 
By \eqref{eq:Emonok}, we see that 
\begin{equation}\label{eq:Ekdercal}
	-\frac{d \cE_k}{d\tau}(\tau) 
	=
	(-t)^\frac{2}{p-1} \int_{\R^n} 
	| S_{(0,0)}[u_k](x,t) |^2 K_{(0,0)}(x,t) dx
\end{equation}
for $\tau>\tau_k$ with $\tau=-\log(-t)$, 
where $S$ is given in Remark \ref{rem:wtSS}. 
By \eqref{eq:dEkto0}, for a.e.~$t\in(-\infty,0)$, we have 
\[
\begin{aligned}
	\int_{\R^n}  | S_{(0,0)}[u_k](x,t) |^2 K_{(0,0)}(x,t) dx 
	\leq C (-t)^{-\frac{2}{p-1}} \leq C_t
\end{aligned}
\]
for $k\geq k_t$. This together with the H\"older inequality 
and $|x|^2 e^{-|x|^2/(4(-t))}\leq C_t$ gives, for a.e.~$t\in(-\infty,0)$, 
\[
\begin{aligned}
	&\int_{B_R}  |\partial_t u_k(x,t)|^2 e^{-\frac{|x|^2}{4(-t)}} dx \\
	&\leq 
	C_t 
	+ C_t \int_{B_R}  |u_k|^2 e^{-\frac{|x|^2}{4(-t)}} dx
	+ C_t \int_{B_R}  |\nabla u_k|^2 |x|^2 e^{-\frac{|x|^2}{4(-t)}} dx 
	\leq C_t
\end{aligned}
\]
for $k\geq k_t$. 
The proof is complete. 
\end{proof}

Finally in this section, 
by using the assumption \eqref{eq:seq}, 
we show a uniform estimate for 
space integrals of $u_k$ at $t=-1$.

\begin{lemma}\label{lem:ukuni1}
Let $R>0$. Then there exists $C>0$ 
depending only on 
$n$, $p$, $M$, $\delta_1$ and $R$ such that, for all $k\geq 1$, 
\[
\begin{aligned}
	\int_{ B_R } (|\nabla u_k(x,-1)|^2 + |u_k(x,-1)|^{p+1} ) dx 
	\leq C. 
\end{aligned}
\]
\end{lemma}

\begin{proof}
We take $(\tilde x,\tilde t)=(0,0)$ as the center of the scaling. 
By \eqref{eq:seq}, the H\"older inequality 
and \eqref{eq:Tkbdd}, we have 
\[
	\int_{\R^n} |u_k(x,-1)|^{p+1} K_{(0,0)}(x,-1) dx 
	\leq C \left( \int_{\R^n} |u_k(x,-1)|^{q_c} dx \right)^\frac{p+1}{q_c}  
	\leq C. 
\]
This implies the desired estimate for $|u_k(\cdot,-1)|^{p+1}$. 
Moreover, \eqref{eq:EkwkcEk} gives 
\[
	\int_{\R^n} |\nabla u_k(x,-1)|^{p+1} K_{(0,0)}(x,-1) dx 
	\leq 2E_{k,(0,0)}(-1) + C. 
\]
By \eqref{eq:EcEk} and Lemma \ref{lem:unibEk}, 
we see that $E_k(-1) = \cE_k(0)  \leq C$. 
Hence the desired estimate for $|\nabla u_k(\cdot,-1)|^2$ also holds. 
\end{proof}

\section{Defect measures}\label{sec:defect}
Let $u$ be a classical solution of \eqref{eq:fujita} satisfying 
\eqref{eq:seq}. Define $u_k$ by \eqref{eq:ukdef}. 
In this section, we show the following strong convergence properties: 

\begin{proposition}\label{pro:stcon}
Suppose \eqref{eq:seq}. 
Then there exist a subsequence $u_k$ 
and a function $\overline{u}$ on $\R^n\times(-\infty,0)$ 
such that 
$u_k\to\overline{u}$ strongly in $L^{p+1}_\loc(\R^n\times(-\infty,0))$ 
and $\nabla u_k\to\nabla \overline{u}$ 
strongly in $L^2_\loc(\R^n\times(-\infty,0))$ as $k\to\infty$, respectively. 
\end{proposition}

To prove the proposition, we first confirm the weak convergence of $u_k$ 
to a blow-up limit $\overline{u}$ in $L^{p+1}_\loc(\R^n\times(-\infty,0))$. 
Next, to examine the difference between the weak convergence and the strong convergence, 
we employ the notion of the (weak-star) defect measures, 
after P. L. Lions \cite{Li84,Li85}, see also DiPerna and Majda \cite[Definition 2.1]{DM88}. 
Motivated by Lin and Wang \cite[Chapters 8, 9]{LWbook} 
and Wang and Wei \cite[Part 1]{WW21pre}, 
we then prove backward self-similarity of the defect measure. 
On the other hand, 
$u_k(\cdot,-1)\to \overline{u}(\cdot,-1)$ strongly in $L^{p+1}_\loc(\R^n)$ 
due to \eqref{eq:seq}. 
Therefore, we finally obtain 
the strong convergence of $u_k$ 
in $L^{p+1}_\loc(\R^n\times(-\infty,0))$. 
The strong convergence of $\nabla u_k$ 
follows from the energy partition argument.

In what follows, we extract subsequences of $u_k$ several times, but 
we also denote by $u_k$ the subsequences. 
We write $\cQ:=\R^n\times (-\infty,0)$. 
Moreover, the term ``a.e.'' means 
almost everywhere with respect to the Lebesgue measure.
We construct a blow-up limit $\overline{u}$ 
as a weak limit of $u_k$ and summarize the topology of 
the convergence $u_k\to\overline{u}$.

\begin{lemma}\label{lem:exbul}
There exist a subsequence $u_k$ 
and a function $\overline{u}$ defined on $\cQ$ 
such that the following properties hold as $k\to\infty$: 
\begin{enumerate}[{\rm (i)}]
\item 
$u_k\to \overline{u}$ weakly in $L^{p+1}_\loc(\cQ)$. 
\item 
$\nabla u_k\to \nabla \overline{u}$ and $\partial_t u_k\to\partial_t \overline{u}$ 
weakly in $L^2_\loc(\cQ)$. 
\item 
$u_k\to \overline{u}$ strongly in $L^q_\loc(\cQ)$ 
for any $1\leq q<p+1$. 
\item 
$\nabla u_k\to \nabla \overline{u}$ strongly in $L^r_\loc(\cQ)$ 
for any $1\leq r<2$. 
\item 
$u_k\to \overline{u}$ and $\nabla u_k\to \nabla \overline{u}$
a.e.~in $\cQ$. 
\end{enumerate}
\end{lemma}

\begin{proof}
(i), (ii), (iii) 
Let $\tilde x\in\R^n$ and $R>0$. 
We write $\cQ_R(\tilde x):=B_R(\tilde x)\times(-2R^2,-R^2)$.  
Lemma \ref{lem:unienk} implies that 
some subsequence $u_k$ converges to a function $\overline{u}$ weakly in 
$L^{p+1}(\cQ_R(\tilde x))$
and that $\nabla u_k\to \nabla \overline{u}$ 
and $\partial_t u_k\to\partial_t \overline{u}$ 
weakly in $L^2(\cQ_R(\tilde x))$, respectively. 
Since $u_k\to \overline{u}$ weakly in $W^{1,2}(\cQ_R(\tilde x))$, 
the Rellich--Kondrachov theorem gives $u_k\to \overline{u}$ 
strongly in $L^1(\cQ_R(\tilde x))$. 
For $1\leq q<p+1$, the standard interpolation shows that 
\[
\begin{aligned}
	\|u_k-\overline{u}\|_{L^q(\cQ_R(\tilde x))}
	&\leq 
	\|u_k-\overline{u}\|_{L^1(\cQ_R(\tilde x))}^\theta 
	\|u_k-\overline{u}\|_{L^{p+1}(\cQ_R(\tilde x))}^{1-\theta} \\
	&\leq 
	C \|u_k-\overline{u}\|_{L^1(\cQ_R(\tilde x))}^\theta \to0
\end{aligned}
\]
for some $0\leq \theta \leq 1$. 
Since $\tilde x$ and $R$ are arbitrary, 
we obtain (i), (ii) and (iii). 

(iv), (v) 
Let $t_1<t_2<t_3<t_4<0$. 
We take $\varphi\in C_0^\infty(\cQ)$ such that $0\leq \varphi\leq 1$, 
$\varphi=1$ in $B_R(\tilde x)\times (t_2,t_3)$ and 
$\varphi=0$ outside $B_{2R}(\tilde x)\times (t_1,t_4)$.  
Set $v_k:=\varphi u_k$. Then, 
\[
\left\{ 
\begin{aligned}
	&\partial_t v_k -\Delta v_k 
	= f_k
	&&\mbox{ in }\R^n\times(-t_1,0), \\
	&v_k(\cdot, -t_1)=0
	&&\mbox{ in }\R^n, 
\end{aligned}
\right.
\]
where 
$f_k:=\varphi |u_k|^{p-1}u_k -(\varphi_t-\Delta \varphi) u_k
+ 2\nabla \varphi \cdot \nabla u_k$. 
By Lemma \ref{lem:unienk}, 
we can check that $\|f_k\|_{L^{(p+1)/p}(\R^n\times(t_1,0))} \leq C$, 
where $C>0$ is independent of $k$. 
Then the maximal regularity (see \cite[Theorem 7.3]{LRbook}) 
guarantees that 
$\|\partial_t v_k\|_{L^{(p+1)/p}(\R^n\times(t_1,0))}
+ \|\Delta v_k\|_{L^{(p+1)/p}(\R^n\times(t_1,0))}\leq C$. 
Thus, 
\[
	\|\partial_t u_k\|_{L^{(p+1)/p}(\cB \times \cI)}+
	\|\Delta u_k\|_{L^{(p+1)/p}(\cB \times \cI)}\leq C, 
\]
where $\cB:=B_R(\tilde x)$ and $\cI:=(t_2,t_3)$. 
By an Aubin--Lions type compactness result 
(see \cite[Appendix B]{MTpre} for instance), we have 
a compact embedding 
\[
\begin{aligned}
	&W^{1,\frac{p+1}{p}}( \cI; L^\frac{p+1}{p} (\cB) )
	\cap L^\frac{p+1}{p}( \cI; W^{2,\frac{p+1}{p}}(\cB) )  \\
	&\hookrightarrow\hookrightarrow 
	L^\frac{p+1}{p} (\cI; (L^\frac{p+1}{p} (\cB), 
	W^{2,\frac{p+1}{p}}(\cB) )_{\frac{1}{2}, \frac{p+1}{p}} ) 
	= 
	L^\frac{p+1}{p} (\cI; B^1_{\frac{p+1}{p}, \frac{p+1}{p}}(\cB) ), 
\end{aligned}
\]
where $\cB:=B_R(\tilde x)$ and 
$(L^{(p+1)/p}, W^{2,(p+1)/p})_{1/2, (p+1)/p}$ 
is the interpolation couple. 
Then by $B^1_{(p+1)/p,(p+1)/p}(\cB)\hookrightarrow W^{1,(p+1)/p}(\cB)$ 
(see \cite[page 327]{Trbook78}), 
$\nabla u_k\to \nabla \overline{u}$ strongly 
in $L^{(p+1)/p}(\cB\times\cI)$. 
Since $\tilde x$, $R$, $t_2$ and $t_3$ are arbitrary, 
the interpolation yields (iv). 
By the strong limits in (iii) and (iv) 
with extracting a subsequence, (v) follows. 
\end{proof}

We also study the limit of $u_k(\cdot,t)$ for a.e.~$t$. 

\begin{lemma}\label{lem:exbult}
Let $\overline{u}$ be as in Lemma \ref{lem:exbul}. 
There exists a subsequence $u_k$ 
such that the following properties hold for a.e.~$t\in(-\infty,0)$ 
as $k\to\infty$: 
\begin{enumerate}[{\rm (i)}]
\item
$u_k(\cdot,t)\to \overline{u}(\cdot,t)$ 
and $\nabla u_k(\cdot,t)\to \nabla \overline{u}(\cdot,t)$
a.e.~in $\R^n$. 
\item
$u_k(\cdot,t)\to \overline{u}(\cdot,t)$ weakly in $L^{p+1}_\loc(\R^n)$. 
\item
$\nabla u_k(\cdot,t)\to \nabla \overline{u}(\cdot,t)$ 
and $\partial_t u_k(\cdot,t)\to \partial_t \overline{u}(\cdot,t)$ 
weakly in $L^2_\loc(\R^n)$. 
\item
$u_k(\cdot,t)\to \overline{u}(\cdot,t)$ strongly in $L^q_\loc(\R^n)$ 
for any $1\leq q<p+1$. 
\end{enumerate}
\end{lemma}

\begin{proof}
(i) This is an immediate consequence of Lemma \ref{lem:exbul} (v). 

(ii), (iii), (iv) 
Let $R>0$. For a.e.~$t\in(-\infty,0)$, 
Lemma \ref{lem:up1fixt} shows that 
some subsequence $u_k(\cdot,t)$ converges to 
a function $\tilde u(\cdot,t)$ weakly in $L^{p+1}(B_R)$
and that $\nabla u_k(\cdot,t)\to \nabla \tilde u(\cdot,t)$ 
and $\partial_t u_k(\cdot,t)\to\partial_t \tilde u(\cdot,t)$ 
weakly in $L^2(B_R)$. 
In particular, $u_k(\cdot,t)\to \tilde u(\cdot,t)$ 
weakly in $W^{1,2}(B_R)$, and so 
$u_k(\cdot,t)\to \tilde u(\cdot,t)$ strongly in $L^1(B_R)$. 
Then the interpolation gives 
$u_k(\cdot,t)\to \tilde u(\cdot,t)$ strongly in $L^q(B_R)$ 
for any $1\leq q<p+1$. 
For a.e.~$t\in (-\infty,0)$, 
some subsequence $u_k(\cdot,t)$ converges to $\tilde u(\cdot,t)$ 
a.e.~in $B_R$. 
By (i) and the uniqueness of the a.e.~convergence, 
we obtain $\tilde u(\cdot,t)=\overline{u}(\cdot,t)$ in $B_R$ 
for a.e.~$t\in (-\infty,0)$. 
Since $R>0$ is arbitrary, we obtain (ii), (iii) and (iv). 
\end{proof}

At $t=-1$, we have a strong limit of $u_k(\cdot,-1)$ 
in $L^{p+1}_\loc(\R^n)$.

\begin{lemma}\label{lem:strt1li}
There exists a subsequence $u_k$ 
such that $u_k(\cdot,-1)\to \overline{u}(\cdot,-1)$ 
strongly in $L^{p+1}_\loc(\R^n)$ as $k\to\infty$. 
\end{lemma}

\begin{proof}
By \eqref{eq:Tkbdd} and by extracting a subsequence, we have 
$u_k(\cdot,-1)\to \overline{u}(\cdot,-1)$ weakly in $L^{q_c}(\R^n)$. 
The weak lower semicontinuity of the norm implies that 
\begin{equation}\label{eq:oubar1M}
	\|\overline{u}(\cdot,-1)\|_{L^{q_c}(\R^n)}
	\leq \liminf_{k\to\infty}\|u_k(\cdot,-1)\|_{L^{q_c}(\R^n)} \leq M. 
\end{equation}
Let $R>0$. 
Lemma \ref{lem:ukuni1} shows that some subsequence $\nabla u_k(\cdot,-1)$ 
converges to $\nabla \overline{u}(\cdot,-1)$ weakly in $L^2(B_R)$. 
Then, $u_k(\cdot,-1)\to\overline{u}(\cdot,-1)$ weakly in $W^{1,2}(B_R)$, 
and thus $u_k(\cdot,-1)\to\overline{u}(\cdot,-1)$ strongly in $L^1(B_R)$. 
By the interpolation with $p+1<q_c$ (since $p>p_S$), 
we see that $u_k(\cdot,-1)\to \overline{u}(\cdot,-1)$ 
strongly in $L^{p+1}(B_R)$. 
Hence the lemma follows. 
\end{proof}

We construct defect measures  
corresponding to the weak limits $u_k\to \overline{u}$. 
In this paper, a measure $\mu$ is called a Radon measure 
if $\mu$ is Borel regular and is locally finite, 
see \cite[Subsection 1.1.3]{EGbook} for details.

\begin{lemma}\label{lem:sptdef}
There exists a subsequence $u_k$ 
such that (i) and (ii) hold: 
\begin{enumerate}[{\rm(i)}]
\item
There exists a nonnegative Radon measure $\mu$ on $\cQ$ 
such that $|u_k|^{p+1}\to|\overline{u}|^{p+1} + \mu$ 
weakly in $\cM(\cQ)$ as $k\to\infty$, that is, 
for any $\varphi\in C_0(\cQ)$, 
\[
	\int_{-\infty}^0 \int_{\R^n} \varphi |u_k|^{p+1} dxdt 
	\to 
	\int_{-\infty}^0 \int_{\R^n} \varphi |\overline{u}|^{p+1} dxdt 
	+ \iint_{\cQ} \varphi d\mu. 
\]
\item
For a.e.~$t\in(-\infty,0)$, 
there exists a nonnegative Radon measure $\mu_t$ on $\R^n$ 
such that $|u_k(\cdot,t)|^{p+1}\to|\overline{u}(\cdot,t)|^{p+1} + \mu_t$ 
weakly in $\cM(\R^n)$ as $k\to\infty$, that is, 
for any $\phi\in C_0(\R^n)$, 
\[
	\int_{\R^n} \phi |u_k(x,t)|^{p+1} dx 
	\to 
	\int_{\R^n} \phi |\overline{u}(x,t)|^{p+1} dx
	+ \int_{\R^n} \phi d\mu_t. 
\]
In addition, $\mu_t$ is defined at $t=-1$ and satisfies
$\mu_{-1}=0$ on $\R^n$. 
\end{enumerate}
\end{lemma}

\begin{proof}
For the assersion (i), 
by Lemma \ref{lem:exbul} (i)
and the weak compactness for Radon measures 
(see \cite[Theorem 1.41]{EGbook}), 
there exist a subsequence $u_k$ 
and a nonnegative Radon measure $\mu$ on $\cQ$ such that 
\[
	\iint_{\cQ} \varphi |u_k-\overline{u}|^{p+1} dxdt 
	\to  
	\iint_{\cQ} \varphi d\mu
\]
for $\varphi\in C_0(\cQ)$ as $k\to\infty$. 
We set $\cQ_\varphi:=\supp \varphi$. 
Since $u_k$ is bounded in $L^{p+1}(\cQ_\varphi)$ by Lemma \ref{lem:exbul} (i)  
and $u_k\to \overline{u}$ a.e.~in $\cQ_\varphi$ by Lemma \ref{lem:exbul} (v), 
we can apply the Brezis-Lieb lemma \cite{BL83}. 
Then the lemma together with the above limit for $|u_k-\overline{u}|^{p+1}$ 
shows that 
\[
\begin{aligned}
	&\left| \iint_{\cQ} \varphi |u_k|^{p+1} dxdt 
	- \iint_{\cQ} \varphi |\overline{u}|^{p+1} dxdt 
	- \iint_{\cQ} \varphi d\mu \right| \\
	&\leq 
	\iint_{\cQ_\varphi}  
	|  |u_k|^{p+1} -  |\overline{u}|^{p+1} -|u_k-\overline{u}|^{p+1}| dxdt   \\
	&\quad 
	+\left| \iint_{\cQ} \varphi |u_k-\overline{u}|^{p+1} dxdt 
	- \iint_{\cQ} \varphi d\mu \right| \to0. 
\end{aligned}
\]
Hence (i) holds. The former part of (ii) 
can be proved in the same way with the help of Lemma \ref{lem:exbult} (i) and (ii). 
The latter part of (ii) follows from the strong convergence 
in Lemma \ref{lem:strt1li}. 
\end{proof}

By modifying the energy partition argument in \cite[Lemma 4.1]{WW21pre}, 
we can show that $\mu_t$ and $\mu$ are also 
defect measures in the weak limits of $\nabla u_k$.

\begin{lemma}\label{lem:enepar}
There exists a subsequence $u_k$ such that 
$|\nabla u_k|^2\to|\nabla \overline{u}|^2 + \mu$ 
weakly in $\cM(\cQ)$ and, for a.e.~$t\in(-\infty,0)$, 
$|\nabla u_k(\cdot,t)|^2\to|\nabla \overline{u}(\cdot,t)|^2 + \mu_t$ 
weakly in $\cM(\R^n)$ as $k\to\infty$. 
\end{lemma}

\begin{proof}
By Lemmas \ref{lem:exbul} and \ref{lem:exbult} 
and the same argument as in Lemma \ref{lem:sptdef}, 
there exist a subsequence $u_k$ and nonnegative Radon measures 
$\tilde \mu$ on $\cQ$ and $\tilde \mu_t$ on $\R^n$ such that 
$|\nabla u_k|^2\to|\nabla \overline{u}|^2 + \tilde \mu$ 
weakly in $\cM(\cQ)$ and, for a.e.~$t\in(-\infty,0)$, 
$|\nabla u_k(\cdot,t)|^2\to|\nabla \overline{u}(\cdot,t)|^2 + \tilde \mu_t$ 
weakly in $\cM(\R^n)$ as $k\to\infty$. 
In what follows, we prove $\mu=\tilde \mu$ and $\mu_t=\tilde \mu_t$ 
for a.e.~$t\in (-\infty,0)$, respectively.

We show that $\mu=\tilde \mu$. 
By the inner and outer regularity of Radon measures 
(see \cite[Theorem 1.8]{EGbook}) and the 
$C^\infty$-Urysohn lemma (see \cite[page 245]{Fobook}), 
it suffices to show that 
\begin{equation}\label{eq:tesUry}
	\iint_\cQ \varphi d\tilde \mu= \iint_\cQ \varphi d \mu 
	\quad \mbox{ for any }\varphi\in C^{2,1}_0(\cQ). 
\end{equation}
From multiplying \eqref{eq:ukfujita} by $\varphi u_k$
and integrating by parts, it follows that 
\[
	\iint u_k \varphi \partial_t u_k 
	+ \iint |\nabla u_k|^2 \varphi 
	+ \iint u_k \nabla \varphi\cdot \nabla u_k
	= \iint |u_k|^{p+1} \varphi, 
\]
where we write $\iint (\cdots)=\iint_\cQ (\cdots) dxdt$ 
when no confusion can arise. 
By the H\"older inequality and Lemma \ref{lem:exbul} (ii) and (iii) with $2<p+1$, 
we have 
\[
\begin{aligned}
	&\left| \iint u_k \varphi\partial_t u_k 
	- \iint \overline{u} \varphi \partial_t \overline{u} \right| \\
	&\leq 
	\left(\iint |u_k-\overline{u}|^2 |\varphi| \right)^\frac{1}{2} 
	\left(\iint |\partial_t u_k|^2 |\varphi| \right)^\frac{1}{2} 
	+\left| \iint \overline{u} \varphi (\partial_t u_k-\partial_t \overline{u}) 
	\right| 
	\to 0
\end{aligned}
\]
Similarly, 
$\iint u_k \nabla \varphi \cdot \nabla u_k
\to \iint \overline{u} \nabla \varphi \cdot \nabla \overline{u}$. 
Then by the definition of $\tilde \mu$ and Lemma \ref{lem:sptdef} (i), 
we obtain 
\[
	\iint \overline{u} \varphi \partial_t \overline{u} 
	+ \iint |\nabla \overline{u}|^2 \varphi 
	+ \iint_\cQ \varphi d\tilde \mu
	+ \iint \overline{u} \nabla \varphi\cdot \nabla \overline{u} 
	= \iint |\overline{u}|^{p+1} \varphi + \iint_\cQ \varphi d \mu. 
\]

We check that 
\begin{equation}\label{eq:wfoou}
	\iint \overline{u} \varphi \partial_t \overline{u} 
	+ \iint |\nabla \overline{u}|^2 \varphi 
	+ \iint \overline{u} \nabla \varphi\cdot \nabla \overline{u} 
	= \iint |\overline{u}|^{p+1} \varphi. 
\end{equation}
Let $\tilde \varphi\in C^{2,1}_0(\cQ)$. 
Multiplying \eqref{eq:ukfujita} by $\tilde \varphi$ and 
integrating by parts yield 
\[
	\iint \tilde \varphi \partial_t u_k
	+\iint \nabla \tilde \varphi \cdot \nabla u_k 
	= \iint |u_k|^{p-1}u_k \tilde \varphi.  
\]
From Lemma \ref{lem:exbul} (iii) with $p<p+1$, 
it follows that 
\[
\begin{aligned}
	&\left| \iint ( |u_k|^{p-1}u_k - |\overline{u}|^{p-1}\overline{u} ) 
	\tilde \varphi \right| \\
	&\leq C \left( \iint |u_k-\overline{u}|^p |\tilde \varphi|\right)^\frac{1}{p} 
	\left( \iint (|u_k|+ |\overline{u}|)^p |\tilde \varphi| \right)^\frac{p-1}{p}
	\to 0. 
\end{aligned}
\]
By letting $k\to\infty$, 
and then by substituting $\varphi u_k$ into $\tilde \varphi$, 
we obtain 
\[
	\iint \varphi u_k \partial_t \overline{u}
	+\iint u_k \nabla \varphi \cdot \nabla \overline{u} 
	+\iint \varphi \nabla u_k \cdot \nabla \overline{u} 
	= \iint |\overline{u}|^{p-1}\overline{u} u_k  \varphi.  
\]
The left-hand side converges to that of \eqref{eq:wfoou} 
by the weak convergence of $u_k$. 
Since $u_k \in L^{p+1}_\loc(\cQ)$ 
and $|\overline{u}|^{p-1} \overline{u}\in L^{(p+1)/p}_\loc(\cQ)$, 
Lemma \ref{lem:exbul} (i) shows that the right-hand side also 
converges to that of \eqref{eq:wfoou}. 
Hence \eqref{eq:wfoou} holds, 
and so \eqref{eq:tesUry} follows. 
Thus, $\mu=\tilde \mu$. 
The relation $\mu_t=\tilde \mu_t$ for a.e.~$t\in (-\infty,0)$ 
can be proved in the same way by using Lemma \ref{lem:exbult}. 
\end{proof}

We prove that the product measure of $\mu_t$ 
and the $1$-dimensional Lebesgue measure coincides 
with the space-time defect measure $\mu$.

\begin{lemma}\label{lem:defect}
For any nonnegative function $\varphi\in C_0(\cQ)$, 
\[
	\iint_\cQ \varphi d\mu = \int_{-\infty}^0\int_{\R^n} \varphi d\mu_t dt. 
\]
In particular, $d\mu= d\mu_t dt$. 
\end{lemma}

\begin{proof}
We set 
\[
	e_k(x,t):= \frac{|\nabla u_k|^2}{2} - \frac{|u_k|^{p+1}}{p+1}, 
	\quad \overline{e}(x,t)
	:= \frac{|\nabla \overline{u}|^2}{2} - \frac{|\overline{u}|^{p+1}}{p+1}. 
\]
Then by Lemmas \ref{lem:sptdef} and \ref{lem:enepar}, we see that 
\begin{align}
	&\left\{ 
	\begin{aligned}
	&\mbox{for a.e.~$t\in(-\infty,0)$,} \quad  \\
	&e_k(\cdot,t) dx \to \overline{e}(\cdot,t) dx + c_p d\mu_t 
	\mbox{ weakly in }\cM(\R^n), \label{eq:ekebt} 
	\end{aligned}
	\right. 
	\\
	&e_k dxdt \to \overline{e} dxdt + c_p d\mu
	\mbox{ weakly in }\cM(\cQ),  
	\label{eq:ekeb}
\end{align}
as $k\to\infty$, where $c_p:=(p-1)/(2(p+1))$. 
Let $\varphi\in C_0(\cQ)$. Define 
\[
	f_k(t):= \int_{\R^n} |\varphi|^2 e_k dx, 
	\quad f_\infty(t):= \int_{\R^n} |\varphi|^2 \overline{e} dx 
	+ c_p \int_{\R^n} |\varphi(x,t)|^2 d\mu_t. 
\]

Let $R>0$ and $t_1<t_2<0$ satisfy $\supp \varphi \subset B_R\times (t_1,t_2)$. 
We claim that 
\begin{equation}\label{eq:kffdom}
	\lim_{k\to\infty} \int_{t_1}^{t_2} f_k(t) dt 
	= \int_{t_1}^{t_2} f_\infty(t) dt. 
\end{equation}
By \eqref{eq:ekebt}, for a.e.~$t\in (t_1,t_2)$, we have 
$f_k(t)\to f_\infty(t)$ as $k\to\infty$. 
We choose $k_\varphi\geq1$ such that $-T_1/T_k<t_1$ 
for any $k\geq k_\varphi$. 
The integration by parts and \eqref{eq:ukfujita} show that 
\[
	\frac{df_k}{dt}(t) 
	= 2\int_{\R^n} \varphi \varphi_t e_k dx 
	- 2\int_{\R^n} \varphi \partial_t u_k \nabla \varphi\cdot \nabla u_k dx 
	- \int_{\R^n} |\varphi|^2 |\partial_t u_k|^2 dx 
\]
and that, for all $k\geq k_\varphi$ and $T_k^{-1}<t''<t'<0$, 
\begin{equation}\label{eq:fkmono}
	f_k(t')-f_k(t'')  
	= \int_{t''}^{t'} \int_{\R^n} 
	(2 \varphi \varphi_t e_k 
	- 2 \varphi \partial_t u_k \nabla \varphi\cdot \nabla u_k  
	-  |\varphi|^2 |\partial_t u_k|^2 ). 
\end{equation}
Here and in the rest of this proof, 
we abbreviate $dxds$ in the integrals.

We give a uniform upper bound of $f_k$. 
Let $t\in (t_1,t_2)$ and $k\geq k_\varphi$. 
From \eqref{eq:fkmono} with $t'=t$ and $t''=-T_1/T_k$, 
the definition of $f_k$ and Young's inequality, it follows that 
\[
\begin{aligned}
	f_k(t) 
	&= f_k\left( -\frac{T_1}{T_k} \right) 
	+ \int_{-\frac{T_1}{T_k}}^t \int_{\R^n} ( 2\varphi \varphi_t e_k 
	- 2 \varphi \partial_t u_k \nabla \varphi\cdot \nabla u_k  
	-  |\varphi|^2 |\partial_t u_k|^2 ) \\
	&\leq 
	C_\varphi \int_{B_R} \left|\nabla u_k \left(x,-\frac{T_1}{T_k} \right)\right|^2 dx
	+ C_\varphi \int_{t_1}^{t_2} \int_{B_R} (|\nabla u_k|^2 + |u_k|^{p+1}), 
\end{aligned}
\]
where $C_\varphi>0$ depends on $\varphi$ and is independent of $k$ and $t$. 
Then, by \eqref{eq:ukdef} and Lemma \ref{lem:unienk}, we have 
\[
\begin{aligned}
	f_k(t) 
	&\leq 
	C_\varphi (-T_k)^{\frac{2}{p-1}+1} 
	\int_{B_R} |\nabla u(\sqrt{-T_k} x,T_1)|^2 dx
	 + C_\varphi \\
	&=: C_\varphi (-T_k)^{\frac{2}{p-1}+1} I_k + C_\varphi. 
\end{aligned}
\]

Let us estimate $I_k$. 
By the change of variables and 
the same computations as in \eqref{eq:nabuT1T1},  
we see that 
\[
\begin{aligned}
	I_k&=
	(-T_k)^{-\frac{n}{2}} \int_{B_{\sqrt{-T_k} R}} |\nabla u(y,T_1)|^2 dy
	\leq 
	C \int_{\R^n} |\nabla u(y,T_1)|^2 K_{(0,0)}(y,T_k) dy \\
	&\leq 
	C ( (-\delta_1 T_1)^{-\frac{n}{2q_c}-\frac{1}{2}} 
	\|u(\cdot,(1+\delta_1)T_1)\|_{L^{q_c}(\R^n)} )^2  \\
	&\quad 
	+ C (-T_k)^{-\frac{n}{2}} \| U_2 \|_{L^{\frac{n(p-1)}{p+1},\infty}(\R^n)}^2  
	\left\| e^{-\frac{|\cdot|^2}{8(-T_k)}} 
	\right\|_{L^{\frac{2n(p-1)}{n(p-1)-2(p+1)},2} (\R^n)}^2, 
\end{aligned}
\]
where $\delta_1$ is the constant in \eqref{eq:cM} 
and $U_2$ is given in \eqref{eq:nabuT1}. 
Then, similar computations in 
the final part of the proof of Lemma \ref{lem:unibE} 
together with $T_1<T_k<0$ show that 
\[
\begin{aligned}
	I_k&\leq CM^2 (-\delta_1 T_1)^{-\frac{2}{p-1}-1}   
	+ CM^{2p} (-T_k)^{-\frac{2}{p-1}-1} \\
	&\leq C( M^2 \delta_1^{-\frac{2}{p-1}-1} +M^{2p} ) 
	(-T_k)^{-\frac{2}{p-1}-1}. 
\end{aligned}
\]
The above computations imply that $f_k(t) \leq C_\varphi$ 
for all $t\in (t_1,t_2)$ and $k\geq k_\varphi$. 
We note that $C_\varphi$ depends on $T_1$, 
but independent of $T_k$ with $k\geq2$.

We show a uniform lower bound of $f_k$. Let $t\in(t_1,t_2)$ and $k\geq k_\varphi$.  
Since $T_k\to0$ as $k\to\infty$, there exists $k'\geq k$ such that 
$t_2<-T_{k'}/T_k<0$. 
By applying \eqref{eq:fkmono} with $t'=-T_{k'}/T_k$ and $t''=t$, 
and then by using $\supp \varphi \subset B_R\times (t_1,t_2)$, we see that 
\[
\begin{aligned}
	f_k(t) 
	&= f_k\left( -\frac{T_{k'}}{T_k} \right) 
	- \int_t^{-\frac{T_{k'}}{T_k}} \int_{\R^n} (2 \varphi \varphi_t e_k 
	- 2 \varphi \partial_t u_k \nabla \varphi\cdot \nabla u_k  
	-  |\varphi|^2 |\partial_t u_k|^2 ) \\
	&\geq 
	- C_\varphi \int_{B_R} 
	\left|u_k \left(x,-\frac{T_{k'}}{T_k} \right)\right|^{p+1} dx 
	- C_\varphi \int_{t_1}^{t_2} \int_{B_R} (|\nabla u_k|^2 + |u_k|^{p+1}). 
\end{aligned}
\]
Switching back to the original variables and 
the H\"older inequality yield 
\[
\begin{aligned}
	f_k(t) &\geq 
	- C_\varphi (-T_k)^\frac{p+1}{p-1} 
	\int_{B_R} |u(\sqrt{-T_k} x, T_{k'})|^{p+1} dx - C_\varphi \\
	&\geq 
	-C_\varphi \left( 
	\int_{B_{\sqrt{-T_k} R}} |u(y,T_{k'})|^{q_c} dy \right)^\frac{p+1}{q_c} 
	-C_\varphi
	\geq  - C_\varphi(M^{p+1}+1). 
\end{aligned}
\]
Thus, $f_k(t)\geq - C_\varphi$. 
By combining the upper bound and the lower bound, we obtain 
$|f_k(t)|\leq C_\varphi$ for all $t\in(t_1,t_2)$ and $k\geq k_\varphi$. 
Hence the Lebesgue dominated convergence theorem shows \eqref{eq:kffdom}, and thus
\[
	\iint_{\cQ} \tilde \varphi e_k dx dt
	\to 
	\iint_{\cQ} \tilde \varphi \overline{e} dx dt
	+ c_p \int_{-\infty}^0\int_{\R^n} \tilde \varphi d\mu_t dt
\]
for any nonnegative $\tilde \varphi\in C_0(\cQ)$ as $k\to\infty$. 
This together with \eqref{eq:ekeb} and the uniqueness of the weak limit in $\cM(\cQ)$ 
shows the lemma. 
\end{proof}

We show the backward self-similarity of $\overline{u}$ 
in the same way as in \cite[Lemma 5.10, Subsection 5.3]{MTpre}, 
see also \cite[Theorem 8.1]{St88}.

\begin{lemma}\label{lem:BSS}
For a.e.~$(x,t)\in \cQ$, 
the blow-up limit $\overline{u}$ satisfies 
\[
	\overline{u}(x,t)=\lambda^\frac{2}{p-1} \overline{u}(\lambda x,\lambda^2 t)
	\quad \mbox{ for any } \lambda>0. 
\]
\end{lemma}

\begin{proof}
Let $r_0>0$ and let $\varphi\in C_0(\cQ)$ satisfy 
$0\leq \varphi\leq 1$ and $\varphi\not \equiv 0$. 
By \eqref{eq:Fkr0cal}, \eqref{eq:Ekdercal} with $\sigma=-\log(-t)$ 
and the H\"older inequality, we see that 
\[
\begin{aligned}
	F_k(r_0) &\geq \frac{1}{2} \int_{-9r_0^2/4}^{-r_0^2} 
	(-t)^{\frac{2}{p-1}-1} \int_{\R^n} | S_{(0,0)}[u_k](x,t) |^2 
	 K_{(0,0)}(x,t) \varphi (x,t) dx dt \\
	&\geq 
	\frac{1}{C} \int_{-9 r_0^2/4}^{-r_0^2} \int_{\R^n} 
	| S_{(0,0)}[u_k](x,t) |^2  \varphi(x,t) dx dt \\ 
	&\geq \frac{1}{C} \left( \int_{-9 r_0^2/4}^{-r_0^2} \int_{\R^n} 
	\left( \frac{u_k}{p-1} + \frac{x}{2} \cdot \nabla u_k 
	+ t  \partial_t u_k \right) \varphi dx dt \right)^2 \geq0, 
\end{aligned}
\]
where $S$ is given in Remark \ref{rem:wtSS} 
and $C$ depends on $r_0$. 
From $F_k(r_0)\to0$ as $k\to\infty$ in \eqref{eq:Fkrto0} 
and the weak convergence properties in 
Lemma \ref{lem:exbul} (i) and (ii), it follows that 
\[
	\int_{-9r_0^2/4}^{-r_0^2} \int_{\R^n} 
	\left( \frac{\overline{u}}{p-1} + \frac{x}{2} \cdot \nabla \overline{u} 
	+ t  \overline{u}_t \right) \varphi dx dt =0. 
\]
Since $r_0>0$ is arbitrary, we obtain  
\[
	\frac{\overline{u}}{p-1} + 
	\frac{x}{2} \cdot \nabla \overline{u} + t  \overline{u}_t  =0
	\quad \mbox{ a.e.~in }\cQ. 
\]
Therefore, for a.e. $(x,t)\in\cQ$, 
\[
\begin{aligned}
	&\lambda^\frac{2}{p-1} \overline{u}(\lambda x,\lambda^2 t) - u(x,t) 
	= \int_1^\lambda \frac{d}{dl} 
	( l^\frac{2}{p-1} \overline{u}(l x,l^2 t) ) dl \\
	&= \int_1^\lambda 
	2 l^{\frac{2}{p-1}-1}\left( 
	\frac{\overline{u}(l x,l^2 t)}{p-1} 
	+ \frac{lx}{2} \cdot \nabla_y \overline{u}(l x,l^2 t) 
	+ l^2 t \overline{u}_s(l x,l^2 t) 
	\right) dl 
	=0 
\end{aligned}
\]
for any $\lambda>0$, and the lemma follows. 
\end{proof}

Let us show that $\mu$ is also backward self-similar in the following sense, 
where a similar statement can be found in \cite[Lemma 5.2]{WW21pre}.

\begin{lemma}\label{lem:BSSmu}
The defect measure $\mu$ satisfies 
\begin{equation}\label{eq:lemBSSmu}
\begin{aligned}
	&\iint_{\cQ} (-t)^\frac{p+1}{p-1}   
	K_{(0,0)}(x,t) \varphi(x,t) d\mu \\
	&=  \lambda^{-2} 
	\iint_{\cQ} (-t)^\frac{p+1}{p-1}   K_{(0,0)}(x,t) 
	\varphi\left( \frac{x}{\lambda}, \frac{t}{\lambda^2} \right) d\mu
\end{aligned}
\end{equation}
for any $\lambda>0$ and $\varphi\in C_0^\infty(\cQ)$ satisfying $0\leq \varphi\leq 1$. 
\end{lemma}

\begin{proof}
Let $\varphi\in C_0^\infty(\cQ)$ satisfy $0\leq \varphi\leq 1$.  
We choose $R>0$ and $t_1<t_2<0$ such that $\supp \varphi \subset B_R\times (t_1,t_2)$. 
For $k\geq1$ and $\lambda>0$, we set 
\[
\begin{aligned}
	g_k(\lambda)
	 &:= \lambda^{-n-2} 
	\iint_{\cQ} (-t)^{\frac{2}{p-1} + 1 } 
	\left( \frac{|\nabla u_k(x,t)|^2}{2} 
	-\frac{|u_k(x,t)|^{p+1}}{p+1}  \right. \\
	&\quad \left. 
	+ \frac{|u_k(x,t)|^2}{2(p-1)(-t)} \right) 
	K_{(0,0)}\left( \frac{x}{\lambda}, \frac{t}{\lambda^2} \right) 
	\varphi\left( \frac{x}{\lambda}, \frac{t}{\lambda^2} \right) dxdt. 
\end{aligned}
\]
From Lemma \ref{lem:exbul} (iii), Lemma \ref{lem:sptdef} (i) and 
Lemma \ref{lem:enepar}, it follows that 
\[
\begin{aligned}
	&\lim_{k\to\infty} g_k(\lambda) 
	= 
	\overline{g}(\lambda)
	+ c_p \lambda^{-n-2} 
	\iint_{\cQ} (-t)^{\frac{2}{p-1} + 1 }  
	K_{(0,0)}\left( \frac{x}{\lambda}, \frac{t}{\lambda^2} \right) 
	\varphi\left( \frac{x}{\lambda}, \frac{t}{\lambda^2} \right) d\mu,  \\
	&\lim_{k\to\infty}g_k(1) =  \overline{g}(1)
	+ c_p \iint_{\cQ} (-t)^\frac{p+1}{p-1}   K_{(0,0)}(x,t) \varphi(x,t) d\mu, 
\end{aligned}
\]
where $c_p:=(p-1)/(2(p+1))$ and 
\[
\begin{aligned}
	\overline{g}(\lambda)&:=\lambda^{-n-2} \iint_{\cQ} (-t)^{\frac{2}{p-1} + 1 } 
	\left( \frac{|\nabla \overline{u}(x,t)|^2}{2} 
	-\frac{|\overline{u}(x,t)|^{p+1}}{p+1} 
	+ \frac{|\overline{u}(x,t)|^2}{2(p-1)(-t)} \right) \\
	&\quad \times 
	K_{(0,0)}\left( \frac{x}{\lambda}, \frac{t}{\lambda^2} \right) 
	\varphi\left( \frac{x}{\lambda}, \frac{t}{\lambda^2} \right) dxdt. 
\end{aligned}
\]

Since the limits of $g_k(\lambda)$ and $g_k(1)$ exist and 
Lemma \ref{lem:BSS} implies $\overline{g}(\lambda)=\overline{g}(1)$, 
it suffices to prove that 
\begin{equation}\label{eq:gklgk1lim}
	\lim_{k\to \infty}|g_k(\lambda)-g_k(1)| =0
\end{equation}
for each $\lambda>0$. 
Define $u_k^\lambda(y,s):= \lambda^{2/(p-1)} u_k(\lambda y, \lambda^2 s)$. 
Then, 
\[
\begin{aligned}
	g_k(\lambda) &=  
	\iint_\cQ (-s)^{\frac{2}{p-1} +1 } 
	\left( \frac{|\nabla u_k^\lambda (y,s)|^2}{2} 
	-\frac{|u_k^\lambda(y,s)|^{p+1}}{p+1} 
	+ \frac{|u_k^\lambda(y,s)|^2}{2(p-1)(-s)} \right) \\
	&\quad \times 
	K_{(0,0)}(y,s)  \varphi( y,s ) dyds. 
\end{aligned}
\]
By differentiating $g_k(\lambda)$ in $\lambda$, 
integrating by parts, using the equation for $u_k^\lambda$, we see that  
\begin{equation}\label{eq:dglcom}
\begin{aligned}
	\frac{dg_k}{d\lambda}
	&= 
	\iint_\cQ (-s)^{\frac{2}{p-1} +1}
	\left( \nabla u_k^\lambda \cdot \nabla \frac{d u_k^\lambda}{d\lambda} 
	- |u_k^\lambda|^{p-1} u_k^\lambda \frac{d u_k^\lambda}{d\lambda} \right. \\
	&\left. \quad + \frac{u_k^\lambda}{(p-1)(-s)}  \frac{d u_k^\lambda}{d\lambda} \right)
	K_{(0,0)} \varphi dyds \\
	&= 
	\iint_\cQ (-s)^{\frac{2}{p-1} } 
	S_{(0,0)}[u_k^\lambda](y,s) 
	\frac{d u_k^\lambda}{d\lambda} K_{(0,0)}(y,s) \varphi(y,s) dyds\\
	&\quad 
	- \iint_\cQ (-s)^{\frac{2}{p-1} +1}
	\frac{d u_k^\lambda}{d\lambda} 
	(\nabla u_k^\lambda\cdot \nabla \varphi)  K_{(0,0)} dyds
	=: I_k(\lambda) - J_k(\lambda), 
\end{aligned}
\end{equation}
where $S$ is given in Remark \ref{rem:wtSS}. 

We compute $I_k(\lambda)$. 
From $\supp \varphi \subset \R^n\times(t_1,t_2)$ and 
\begin{equation}\label{eq:uklamcal}
\left\{ 
\begin{aligned}
	&\frac{d u_k^\lambda}{d\lambda} (y,s) 
	= 2 \lambda^{\frac{2}{p-1}-1} 
	S_{(0,0)} [u_k] (\lambda y,\lambda^2 s), 
	\\ 
	&\partial_s u_k^\lambda(y,s) 
	= \lambda^{\frac{2}{p-1}+2} \partial_t u_k(\lambda y,\lambda^2 s), \\
	&\nabla u_k^\lambda(y,s) 
	= \lambda^{\frac{2}{p-1}+1} \nabla u_k(\lambda y,\lambda^2 s), 
\end{aligned}
\right.
\end{equation}
it follows that 
\[
\begin{aligned}
	I_k(\lambda) &= 
	\frac{2}{\lambda} \int_{t_1}^{t_2} \int_{\R^n} (-\lambda^2 s)^\frac{2}{p-1} 
	| S_{(0,0)}[u_k](\lambda y, \lambda^2 s) |^2 
	K_{(0,0)}(y,s) \varphi(y,s) dyds \\
	&= 
	\frac{2}{\lambda^3} \int_{\lambda^2 t_1}^{\lambda^2 t_2} \int_{\R^n} 
	(-t)^\frac{2}{p-1} 
	| S_{(0,0)}[u_k](x,t) |^2 
	K_{(0,0)}(x,t)\varphi\left( \frac{x}{\lambda}, \frac{t}{\lambda^2} \right) dxdt. 
\end{aligned}
\]
In particular, we observe that $I_k(\lambda)\geq0$. 
Since $0\leq \varphi \leq1$, we obtain 
\begin{equation}\label{eq:tilIlamk}
\begin{aligned}
	0&\leq I_k(\lambda) \leq \tilde I_k(\lambda)  \\
	&:=\frac{2}{\lambda^3} \int_{\lambda^2 t_1}^{\lambda^2 t_2} \int_{\R^n} 
	(-t)^\frac{2}{p-1}  | S_{(0,0)}[u_k](x,t) |^2 
	K_{(0,0)}(x,t) dxdt. 
\end{aligned}
\end{equation}
We note that $\tilde I_k(\lambda)$ will be used in 
the estimate for $J_k(\lambda)$. 
By \eqref{eq:Ekdercal} with $\tau=-\log(-t)$ and 
$\cE_k(\tau)=E_k(t)$, we have 
\[
\begin{aligned}
	0&\leq I_k(\lambda) \leq \tilde I_k(\lambda) \leq 
	\frac{2}{\lambda^3} 
	\int_{-\log(-\lambda^2 t_1)}^{-\log(-\lambda^2 t_2)} 
	\left( -\frac{d \cE_k}{d\tau}(\tau) \right) 
	e^{-\tau} d\tau \\
	&\leq 
	\frac{-2t_1}{\lambda} 
	\int_{-\log(-\lambda^2 t_1)}^{-\log(-\lambda^2 t_2)} 
	\left( -\frac{d \cE_k}{d\tau}(\tau) \right) d\tau
	= \frac{-2t_1}{\lambda} ( E_k(\lambda^2 t_1) - E_k(\lambda^2 t_2) ). 
\end{aligned}
\]
In what follows, we consider the case $\lambda>1$. 
The case $0<\lambda<1$ can be handled similarly. 
As in the proof of \eqref{eq:Fkrto0}, 
letting $k\to \infty$ yields 
\begin{equation}\label{eq:IktIkto0}
\begin{aligned}
	0&\leq \int_1^\lambda I_k(l) dl 
	\leq \int_1^\lambda \tilde I_k(l) dl \\
	&\leq 
	-2t_1 \int_1^\lambda 
	\frac{1}{l} ( E_k(l^2 t_1) - E_k(l^2 t_2) ) dl 
	\to 0. 
\end{aligned}
\end{equation}

As for $J_k(\lambda)$, by $\supp \varphi \subset B_R\times (t_1,t_2)$ and 
\eqref{eq:uklamcal},  we compute that 
\[
\begin{aligned}
	J_k(\lambda) 
	&= \frac{2}{\lambda^2} 
	\int_{t_1}^{t_2} \int_{B_R} (-\lambda^2 s)^{\frac{2}{p-1}+1} 
	S_{(0,0)}[u_k](\lambda y, \lambda^2 s)  \\
	&\quad \times 
	(  \nabla u_k(\lambda y,\lambda^2 s) 
	\cdot \nabla \varphi(y,s) ) 
	K_{(0,0)}(y,s) dyds \\
	&= 
	\frac{2}{\lambda^4} 
	\int_{\lambda^2 t_1}^{\lambda^2 t_2} \int_{B_{\lambda R}} 
	(-t)^{\frac{1}{p-1}} S_{(0,0)}[u_k](x, t)  \\
	&\quad \times 
	(-t)^{\frac{1}{p-1}+1} 
	\left( \nabla u_k(x,t) 
	\cdot \nabla \varphi\left( \frac{x}{\lambda}, \frac{t}{\lambda^2} \right)
	\right)  K_{(0,0)}(x,t) dxdt. 
\end{aligned}
\]
The H\"older inequality and the definition of $\tilde I_k(\lambda)$ 
in \eqref{eq:tilIlamk} show that 
\[
\begin{aligned}
	|J_k(\lambda)| 
	&\leq  
	\frac{C \tilde I_k^{1/2}}{\lambda^{5/2}} 
	\left( \int_{\lambda^2 t_1}^{\lambda^2 t_2} \int_{B_{\lambda R}} 
	(-t)^{\frac{2}{p-1}+2} 
	|\nabla u_k|^2  \left| \nabla \varphi\left( \frac{x}{\lambda}, \frac{t}{\lambda^2} 
	\right) \right|^2 
	K_{(0,0)} dxdt\right)^\frac{1}{2} \\
	&\leq 
	\frac{C_\varphi }{\lambda^{5/2}} 
	\left( \tilde I_k(\lambda) \right)^\frac{1}{2} 
	\left( \lambda^{ \frac{4}{p-1}+4-n }
	\int_{\lambda^2 t_1}^{\lambda^2 t_2} \int_{B_{\lambda R}} 
	|\nabla u_k(x,t)|^2 dxdt\right)^\frac{1}{2}, 
\end{aligned}
\]
where $C_\varphi$ depends on $\varphi$ 
and is independent of $k$ and $\lambda$. 
We note that there exists $\tilde k_1\geq1$ 
depending only on $T_1$ and $t_1$ such that 
$T_1/(-t_1)<T_k<0$ for $k\geq \tilde k_1$. 
Then, 
$T_1/(-t_1\lambda^2)<T_1/(-t_1)<T_k$ 
holds for any $k\geq \tilde k_1$ and $\lambda>1$, 
and so Lemma \ref{lem:unienk} gives 
\[
	|J_k(\lambda)| \leq 
	C_\varphi \lambda^{-\frac{1}{2}} 
	( \tilde I_k(\lambda) )^\frac{1}{2} 
\]
for $k\geq \tilde k_1$ and $\lambda>1$. 
Thus, \eqref{eq:IktIkto0} yields 
\[
	0\leq \int_1^\lambda |J_k(l)| dl 
	\leq 
	C_\varphi \int_1^\lambda
	( \tilde I_k(l) )^\frac{1}{2} dl  
	\leq 
	C_\varphi (\lambda-1)^\frac{1}{2} 
	\left( \int_1^\lambda
	\tilde I_k(l)  dl \right)^\frac{1}{2} 
	\to 0 
\]
as $k\to\infty$. 
This together with \eqref{eq:dglcom} and \eqref{eq:IktIkto0} shows that 
\[
	|g_k(\lambda)-g_k(1)|
	=\left| \int_1^\lambda I_k(l) dl 
	- \int_1^\lambda J_k(l) dl \right|  
	\to 0
\]
for $\lambda>1$ as $k\to\infty$. 
In particular, \eqref{eq:gklgk1lim} holds for $\lambda>1$. 
Similarly, \eqref{eq:gklgk1lim} also holds for  $0<\lambda<1$. 
Then the lemma follows. 
\end{proof}

We are now in a position to prove Proposition \ref{pro:stcon}. 

\begin{proof}[Proof of Proposition \ref{pro:stcon}]
Let $\lambda>0$. 
For $R_2>R_1>0$ and $t_1<t_2<t_3<t_4<0$, 
we choose $\varphi\in C^\infty_0(\cQ)$ such that 
$0\leq \varphi\leq 1$, $\varphi=1$ on $B_{R_1}\times(t_2,t_3)$ 
and $\varphi=0$ outside $B_{R_2}\times(t_1,t_4)$.  
Then, Lemma \ref{lem:defect} and Fubini's theorem show that 
the left-hand side of \eqref{eq:lemBSSmu} can be estimated as 
\[
\begin{aligned}
	\iint_{\cQ} (-t)^\frac{p+1}{p-1}   
	K_{(0,0)}(x,t) \varphi(x,t) d\mu
	&\geq 
	\int_{t_2}^{t_3}\int_{B_{R_1}} (-t)^{\frac{p+1}{p-1}-\frac{n}{2} }   
	e^{-\frac{|x|^2}{4(-t)}} \varphi(x,t) d\mu_t dt \\
	&\geq 
	\int_{t_2}^{t_3} (-t)^{\frac{p+1}{p-1}-\frac{n}{2} } 
	e^{-\frac{R_1^2}{4(-t)}}
	\mu_t(B_{R_1}) dt. 
\end{aligned}
\]
On the other hand, for the right-hand side of \eqref{eq:lemBSSmu}, we have 
\[
\begin{aligned}
	&\lambda^{-2} 
	\iint_{\cQ} (-t)^\frac{p+1}{p-1}  K_{(0,0)}(x,t) 
	\varphi\left( \frac{x}{\lambda}, \frac{t}{\lambda^2} \right) d\mu \\
	&\leq 
	\lambda^{-2} 
	\int_{\lambda^2 t_1}^{\lambda^2 t_4} \int_{B_{\lambda R_2}} 
	(-t)^{\frac{p+1}{p-1}-\frac{n}{2} }  d\mu_t dt 
	= 
	\lambda^{-2} 
	\int_{\lambda^2 t_1}^{\lambda^2 t_4} 
	(-t)^{\frac{p+1}{p-1}-\frac{n}{2} } \mu_t(B_{\lambda R_2}) dt \\
	&= 
	\lambda^{\frac{2(p+1)}{p-1}-n }
	\int_{t_1}^{t_4} 
	(- s)^{\frac{p+1}{p-1}-\frac{n}{2} } 
	\mu_{\lambda^2 s}(B_{\lambda R_2}) ds. 
\end{aligned}
\]

The above estimates together with  Lemma \ref{lem:BSSmu} imply that 
\[
	\int_{t_2}^{t_3} (-t)^{\frac{p+1}{p-1}-\frac{n}{2} } 
	e^{-\frac{R_1^2}{4(-t)}}
	\mu_t(B_{R_1}) dt
	\leq 
	\lambda^{\frac{2(p+1)}{p-1}-n }
	\int_{t_1}^{t_4} 
	(- t)^{\frac{p+1}{p-1}-\frac{n}{2} } 
	\mu_{\lambda^2 t}(B_{\lambda R_2}) dt 
\]
and that 
\[
\begin{aligned}
	&\frac{1}{2\delta} \int_{t-\delta}^{t+\delta} 
	(-t)^{\frac{p+1}{p-1}-\frac{n}{2} } 
	e^{-\frac{R_1^2}{4(-t)}}
	\mu_t(B_{R_1}) dt \\
	&\leq 
	2\lambda^{\frac{2(p+1)}{p-1}-n }
	\frac{1}{4\delta}
	\int_{t-2\delta}^{t+2\delta} 
	(- t)^{\frac{p+1}{p-1}-\frac{n}{2} } 
	\mu_{\lambda^2 t}(B_{\lambda R_2}) dt 
\end{aligned}
\]
for any $t<0$ and $\delta>0$ satisfying $t+2\delta<0$. 
Letting $\delta\to0$ with the aid of 
the Lebesgue differentiation theorem gives, for a.e.~$t\in(-\infty,0)$, 
\[
	e^{-\frac{R_1^2}{4(-t)}}
	\mu_t(B_{R_1}) 
	\leq 
	2\lambda^{\frac{2(p+1)}{p-1}-n }
	\mu_{\lambda^2 t}(B_{\lambda R_2})  
\]
for any $\lambda>0$. 
For each $t\in (-\infty,0)$ satisfying this inequality, 
we take $\lambda=(-t)^{-1/2}$. 
Since $\mu_{-1}=0$ on $\R^n$ as in Lemma \ref{lem:sptdef} (ii), we have 
\[
	\mu_t(B_{R_1}) \leq  
	2 (-t)^{\frac{n}{2}-\frac{p+1}{p-1}} e^\frac{R_1^2}{4(-t)} 
	\mu_{-1}(B_{R_2/\sqrt{-t}}) =0
\]
for a.e.~$t\in (-\infty,0)$. 
This implies that 
$\mu(B_{R_1}\times (-\infty,0))=\int_{-\infty}^0 \mu_t(B_{R_1}) dt=0$ 
for any $R_1>0$, 
and so $\mu=0$ on $\cQ$. 
By Lemma \ref{lem:sptdef} (i), 
we obtain $u_k\to\overline{u}$ strongly in $L^{p+1}_\loc(\cQ)$.  
Lemma \ref{lem:enepar} also yields 
$\nabla u_k\to\nabla \overline{u}$ 
strongly in $L^2_\loc(\cQ)$ as $k\to\infty$. 
The proof is complete. 
\end{proof}

\section{Improved $\eps$-regularity}\label{sec:eps}
We prove the following one-scale $\eps$-regularity theorem 
for the original solution $u$, and then 
we also show an analogous $\eps$-regularity 
for the blow-up limit. 
In what follows, we write 
$Q_r(x,t):= B_r(x)\times (t-r^2,t)$ and $Q_r:=Q_r(0,0)$ 
for $(x,t)\in\R^{n+1}$ and $r>0$. 
We recall that $\delta_1$ is the constant given in 
\eqref{eq:cM}. 

\begin{theorem}\label{th:epsreg}
Let $n\geq3$, $p>p_S$ and $u$ be a classical $L^{q_c}$-solution
of \eqref{eq:fujita} satisfying \eqref{eq:seq}. 
Then there exist constants $0<\eps_0<1$, $A>1$ and $C>0$ 
depending only on $n$, $p$, $M$ and $\delta_1$ such that  
the following holds: 
If 
\begin{equation}\label{eq:1scale}
\left\{ 
\begin{aligned}
	&\delta^{\frac{4}{p-1}-n} 
	\int_{-\delta^2}^{-\delta^2/2} \int_{B_{A\delta}(x_0)} 
	 |u(x,t)|^{p+1} dxdt \leq \eps_0 \\
	&\mbox{for some }x_0\in\R^n \mbox{ and }
	0<\delta\leq \sqrt{-T_1}, 
\end{aligned} \right. 
\end{equation}
then 
\[
	\|u\|_{L^\infty(Q_{\delta/16}(x_0,0))} 
	\leq C\delta^{-\frac{2}{p-1}}. 
\]
\end{theorem}

In \cite[Theorem 4.1]{MTpre}, we proved a similar $\eps$-regularity theorem 
based on the argument of Chou, Du and Zheng \cite[Theorems 2, 2']{CDZ07} 
under the assumption that 
the $L^{q_c}$ norm of $u$ is uniformly bounded for $t\in(-1,0)$ and 
that integral quantities concerning both $|u|^{p+1}$ and $|\nabla u|^2$ are small. 
Now we weaken the assumption on the $L^{q_c}$ norm and 
the smallness of $u$.

By translation invariance, it suffices to prove the case $x_0=0$. 
For $\tilde x\in \R^n$, $-1<\tilde t\leq0$ and $\rho>0$, 
we set 
\[
	I_\rho=I_\rho[u](\tilde x,\tilde t):= 
	(\rho/2)^{\frac{4}{p-1}-n}
	\int_{\tilde t-\rho^2/4}^{\tilde t-\rho^2/16} \int_{B_{\rho/2}(\tilde x)}
	|u(x,t)|^{p+1} dxdt. 
\]
We recall the following multi-scale $\eps$-regularity in \cite{MTpre}: 

\begin{lemma}\label{lem:epspre}
There exist $\eps_1>0$ and $C>0$ 
depending only on $n$ and $p$ such that 
the following holds: 
If there exists $\delta>0$ such that 
$I_\rho[u](\tilde x,\tilde t)\leq \eps_1$ 
for any $\rho>0$, $\tilde x\in\R^n$ and $-1<\tilde t\leq 0$ 
satisfying $Q_\rho(\tilde x,\tilde t)\subset Q_{\delta/2}$, 
then $\|u\|_{L^\infty(Q_{\delta/16})}\leq C\delta^{-2/(p-1)}$. 
\end{lemma}

\begin{proof} 
This lemma can be proved by applying the $\eps$-regularity in the 
parabolic Morrey space \cite[Proposition 4.1]{BS15}, 
see \cite[Lemma 4.4]{MTpre} for details. 
\end{proof}

Under the assumption \eqref{eq:1scale} with $x_0=0$, 
we derive estimates of $I_\rho$, 
where constants $0<\eps_0<1$ and $A>1$ 
will be determined later.

\begin{lemma}\label{lem:32}
There is $C>0$ depending only on 
$n$, $p$, $M$ and $\delta_1$ such that 
\[
	I_\rho(\tilde x,\tilde t) \leq Ch( \eps_0
	+ \cI_\delta(\tilde x,\tilde t) ) 
\]
for $\rho>0$, $\tilde x\in\R^n$ and $-1<\tilde t\leq 0$ 
with $Q_\rho(\tilde x,\tilde t)\subset Q_{\delta/2}$, 
where $h(s):=s+s^{1/2(p+1)}$ for $s\geq0$ 
and 
\[
	\cI_\delta=\cI_\delta(\tilde x,\tilde t):= 
	\delta^\frac{4}{p-1} 
	\int_{-\delta^2}^{-\delta^2/2} 
	\int_{\R^n\setminus B_{A\delta}} |u(x,t)|^{p+1} K_{(\tilde x,\tilde t)}(x,t) dxdt. 
\]
\end{lemma}

\begin{proof}
We note that $|\tilde x|\leq \delta/2$, $-\delta^2/4\leq \tilde t\leq 0$ 
and $\rho\leq \delta/2$ hold if $Q_\rho(\tilde x,\tilde t) \subset Q_{\delta/2}$. 
By using the backward similarity variables, 
applying Lemma \ref{lem:wp1ene} with $\tau=-\log(\rho^2/16)$ 
and $\tau'=-\log(\rho^2/4)$ 
and using \eqref{eq:EcE} with 
the monotonicity of $E$, we see that 
\[
\begin{aligned}
	I_\rho(\tilde x,\tilde t) &\leq 
	C \int_{\tilde t-\rho^2/4}^{\tilde t-\rho^2/16} 
	(\tilde t-t)^\frac{2}{p-1} 
	\int_{B_{\rho/2}(\tilde x)} |u(x,t)|^{p+1} (\tilde t-t)^{-\frac{n}{2}} 
	e^{-\frac{|x-\tilde x|^2}{4(\tilde t-t)}}
	dxdt \\
	&\leq 
	C \int_{-\log(\rho^2/4)}^{-\log(\rho^2/16)} \int_{\R^n} 
	|w_{(\tilde x,\tilde t)}(\eta,\tau)|^{p+1} \rho(\eta) d\eta d\tau \\
	&\leq 
	C f( \cE_{(\tilde x,\tilde t)} (  -\log (\rho^2/4) ) ) f(\log 4)
	\leq C f( E_{(\tilde x,\tilde t)} (  \tilde t-\rho^2/4 ) ) \\
	&\leq 
	C f( E_{(\tilde x,\tilde t)} ( - \delta^2/2 ) ) 
	\leq 
	Cf\left( \frac{1}{\delta^2}
	\int_{-\delta^2}^{-\delta^2/2} 
	E_{(\tilde x,\tilde t)}(t) dt \right)
\end{aligned}
\]
if $Q_\rho(\tilde x,\tilde t)\subset Q_{\delta/2}$, 
where $f(s)=s+s^{1/2}$ for $s\geq0$. 
From Lemma \ref{lem:Ep1} with $\tau=-\log(\tilde t + \delta^2/2)$, 
$\tau'=-\log(\tilde t+\delta^2)$ and 
$\log((\tilde t+\delta^2)/(\tilde t+\delta^2/2)) \leq C$, 
it follows that 
\[
\begin{aligned}
	&\frac{1}{\delta^2} \int_{-\delta^2}^{-\delta^2/2} 
	E_{(\tilde x,\tilde t)}(t) dt
	\leq \frac{1}{\delta^2}
	\int_{-\log(\tilde t+\delta^2)}^{-\log(\tilde t+\delta^2/2)} 
	e^{-\tau} 
	\cE_{(\tilde x,\tilde t)}(\tau) d\tau  \\
	&\leq 
	C \int_{-\log(\tilde t+\delta^2)}^{-\log(\tilde t+\delta^2/2)} 
	\cE_{(\tilde x,\tilde t)}(\tau) d\tau  
	\leq 
	C \tilde f\left( 
	\int_{-\log(\tilde t+\delta^2)}^{-\log(\tilde t+\delta^2/2)}  \int_{\R^n} 
	|w_{(\tilde x,\tilde t)}|^{p+1} \rho d\eta d\tau \right) \\
	&\leq 
	C\tilde f\left( 
	\int_{-\delta^2}^{-\delta^2/2} 
	(\tilde t-t)^\frac{2}{p-1} 
	\int_{\R^n} |u(x,t)|^{p+1} K_{(\tilde x,\tilde t)}(x,t) dxdt \right), 
\end{aligned}
\]
where $\tilde f(s):=s+s^{1/(p+1)}$ for $s\geq0$. 
Thus, 
\[
	I_\rho(\tilde x,\tilde t)
	\leq 
	C(f\circ \tilde f)\left( 
	\delta^{\frac{4}{p-1}-n}
	\int_{-\delta^2}^{-\delta^2/2} 
	\int_{B_{A\delta}} |u|^{p+1} dxdt 
	+ \cI_\delta(\tilde x,\tilde t) \right). 
\]
Then \eqref{eq:1scale} yields the desired inequality.  
\end{proof}

We estimate $\cI_\delta$.

\begin{lemma}\label{lem:Idout}
There is $C>0$ depending only on 
$n$, $p$, $M$ and $\delta_1$ such that 
\[
	\cI_\delta(\tilde x,\tilde t) \leq C e^{-\frac{A^2}{C}}
\]
for $\rho>0$, $\tilde x\in\R^n$ and $-1<\tilde t\leq 0$ 
with $Q_\rho(\tilde x,\tilde t)\subset Q_{\delta/2}$. 
\end{lemma}

\begin{proof}
We use the notation $K_{(\tilde x,\tilde {t})}$ 
even for $\tilde t>0$. Then,  
\[
	K_{(\tilde x,\tilde t)} (x,t) 
	\leq 
	CK_{(0,\delta^2/2)} (x,t)
	\leq 
	Ce^{-\frac{A^2}{C}} K_{(0,\delta^2)} (x,t)
\]
for $x\in \R^n\setminus B_{A\delta}$ and $-\delta^2 < t < -\delta^2/2$ 
if $Q_\rho(\tilde x,\tilde t)\subset Q_{\delta/2}$. 
Thus, 
\begin{equation}\label{eq:fIddef}
\begin{aligned}
	&\cI_\delta(\tilde x,\tilde t) 
	\leq 
	\delta^\frac{4}{p-1} e^{-\frac{A^2}{C}}
	\int_{-\delta^2}^{-\delta^2/2}  
	\int_{\R^n} |u(x,t)|^{p+1} K_{(0,\delta^2)}(x,t) dxdt \\
	&\leq 
	C e^{-\frac{A^2}{C}}
	\int_{-\delta^2}^{-\delta^2/2}  
	(\delta^2 - t)^\frac{2}{p-1}
	\int_{\R^n} |u|^{p+1} K_{(0,\delta^2)} dxdt 
	=: C e^{-\frac{A^2}{C}} \cJ_\delta. 
\end{aligned}
\end{equation}

We observe that $E_{(\tilde x,\tilde t)}(t)$, 
$w_{(\tilde x, \tilde t)}(\eta,\tau)$ and $\cE_{(\tilde x,\tilde t)}(\tau)$ 
can be defined 
even when $\tilde t>0$ provided that $-1<t<0$ and 
$-\log(\tilde t+1)<\tau<-\log \tilde t$. 
In the case $\tilde t>0$, the properties of $\cE$ 
which rely on the infinite length of the time-interval for $\tau$, 
such as \eqref{eq:wno2},  could not hold. 
However, the monotonicity formula \eqref{eq:GKmn} still holds. 
Therefore, from \eqref{eq:w22} and 
the same computations as in Lemmas \ref{lem:Ep1} and \ref{lem:wp1ene} 
avoiding the use of \eqref{eq:wno2}, 
it follows that 
\begin{equation}\label{eq:Ep1app}
\begin{aligned}
	&\int_{\tau'}^\tau \int_{\R^n} |w_{(\tilde x,\tilde t)}(\eta,\sigma)|^{p+1} 
	\rho(\eta) d\eta d\sigma \\
	&\leq 
	C( \cE_{(\tilde x,\tilde t)}(\tau')
	-\cE_{(\tilde x,\tilde t)}(\tau) )^\frac{1}{2}  
	(\tau-\tau')^\frac{p-1}{2(p+1)}  \\
	&\quad \times 
	\left( \int_{\tau'}^\tau \int_{\R^n} |w_{(\tilde x,\tilde t)}|^{p+1}
	\rho d\eta d\sigma 
	\right)^\frac{1}{p+1} 
	+ \frac{2(p+1)}{p-1} \cE_{(\tilde x,\tilde t)}(\tau') (\tau-\tau')
\end{aligned}
\end{equation}
for $\tilde t>0$ and $-\log(\tilde t+1)<\tau'<\tau<-\log \tilde t$. 
Here we remark that $\cE$ could be negative, since 
we do not know the validity of \eqref{eq:wno2}. 
Moreover, the proof of Lemma \ref{lem:unibE} is still valid when $0<\tilde t<-T_1$, 
since $-\delta_1 T_1\geq 2^{-1} \delta_1 (\tilde t-T_1)$. 
Thus, the uniform estimate 
\begin{equation}\label{eq:cEbddaf}
	\cE_{(\tilde x,\tilde t)}(\tau) \leq C 
\end{equation}
also holds for $0<\tilde t<-T_1$ and $-\log(\tilde t-T_1)<\tau<-\log \tilde t$.

By the assumption on $\delta$ in \eqref{eq:1scale}, we can check that 
$0<\delta^2<-T_1$ and $-\log(\delta^2-T_1)<-\log(2\delta^2)<-\log(\delta^2)$, 
and so $\cE_{(0,\delta^2)}(-\log(2\delta^2)) \leq C$ holds. 
Therefore, \eqref{eq:Ep1app} and \eqref{eq:cEbddaf} yield 
\begin{equation}\label{eq:tilIdele}
\begin{aligned}
	\cJ_\delta
	&= 
	\int_{-\log(2\delta^2)}^{-\log(3\delta^2/2)}  
	\int_{\R^n} |w_{(0,\delta^2)}(\eta,\tau)|^{p+1} \rho(\eta) 
	d\eta d\tau \\
	&\leq 
	C( C -\cE_{(0,\delta^2)}(-\log(3\delta^2/2 ) ) )^\frac{1}{2}  
	\cJ_\delta^\frac{1}{p+1} + C. 
\end{aligned}
\end{equation}
Recall \eqref{eq:seq}. 
Then there exists $k'\geq1$ depending on $\delta$ 
such that $-\delta^2/2<T_{k'}<0$. 
Since 
$-\log(\delta^2+1) < -\log( 3\delta^2/2 ) 
< -\log(\delta^2-T_{k'} ) < -\log (\delta^2)$, 
we can apply the monotonicity of $\cE_{(0,\delta^2)}$ to see that 
\[
	-\cE_{(0,\delta^2)}(-\log(3\delta^2/2 ) )
	\leq 
	-\cE_{(0,\delta^2)}(-\log(\delta^2-T_{k'}  )) 
	= 
	-E_{(0,\delta^2)}(T_{k'}). 
\]
From \eqref{eq:oriEdef}, \eqref{eq:seq}
and the H\"older inequality, it follows that 
\begin{equation}\label{eq:unibElow}
\begin{aligned}
	&-E_{(0,\delta^2)}(T_{k'}) \\
	&\leq 
	C (\delta^2-T_{k'})^{\frac{2}{p-1}+1} 
	\int_{\R^n} |u(x,T_{k'})|^{p+1}  K_{(0,\delta^2)}(x,T_{k'}) dx \\
	&\leq 
	CM^{p+1} (\delta^2-T_{k'})^{\frac{p+1}{p-1}-\frac{n}{2}}  
	\left( \int_{\R^n} e^{-\frac{|x|^2}{C(\delta^2-T_{k'})}} dx 
	\right)^{1-\frac{p+1}{q_c}} \leq CM^{p+1}. 
\end{aligned} 
\end{equation}
Hence by \eqref{eq:tilIdele}, there exists $C>0$ 
depending only on  $n$, $p$, $M$ and $\delta_1$ such that 
$\cJ_\delta \leq C \cJ_\delta^{1/(p+1)} + C$. 
Since the right-hand side is sublinear, we obtain 
$\cJ_\delta \leq C$. 
This together with \eqref{eq:fIddef} implies 
the desired estimate. 
\end{proof}

We are now in a position to prove Theorem \ref{th:epsreg}.

\begin{proof}[Proof of Theorem \ref{th:epsreg}]
By Lemmas \ref{lem:32} and \ref{lem:Idout}, 
there exists a constant $\tilde C>0$ satisfying 
$I_\rho(\tilde x,\tilde t) \leq 
\tilde C h(\eps_0) + \tilde C h(e^{-A^2/\tilde C})$ 
for any $0<\eps_0<1$ and $A>1$ if $Q_\rho(\tilde x,\tilde t)\subset Q_{\delta/2}$, 
where $\tilde C$ is independent of $\delta$. 
For $\eps_1$ in Lemma \ref{lem:epspre}, 
we choose $0<\eps_0<1$ and $A>1$ such that 
$\tilde C  h(\eps_0)<\eps_1/2$ and 
$\tilde C h( e^{-A^2/\tilde C} ) <\eps_1/2$. 
Then, $I_\rho(\tilde x,\tilde t)\leq \eps_1$ 
if $Q_\rho(\tilde x,\tilde t)\subset Q_{\delta/2}$. 
Hence we can apply Lemma \ref{lem:epspre} to obtain 
$\|u\|_{L^\infty( Q_{\delta/16}) } \leq C\delta^{-2/(p-1)}$. 
The proof is complete. 
\end{proof}

In a similar way to Theorem \ref{th:epsreg}, 
we also show the following $\eps$-regularity for the blow-up limit $\overline{u}$ 
given in Proposition \ref{pro:stcon}:

\begin{lemma}\label{lem:blepsreg}
There exist constants $0<\eps_*<1$, $A>1$ and $C>0$ 
depending only on $n$, $p$, $M$ and $\delta_1$ such that the following holds: 
If 
\begin{equation}\label{eq:1scalebl}
\begin{aligned}
	\delta^{\frac{4}{p-1}-n} 
	\int_{t_0-\delta^2}^{t_0-\delta^2/2} \int_{B_{A\delta}(x_0)} 
	|\overline{u}|^{p+1} dxdt 
	\leq \eps_*  
\end{aligned}
\end{equation}
for some $0<\delta\leq \sqrt{-T_1}$ and $(x_0,t_0)\in\R^n\times(-\infty,0]$, 
then 
\[
	\|\overline{u}\|_{L^\infty(Q_{\delta/16}(x_0,t_0))} 
	\leq C\delta^{-\frac{2}{p-1}}. 
\]
\end{lemma}

\begin{proof}
By translation invariance, we only have to consider the case 
$(x_0,t_0)=(0,0)$. 
The strategy of proof is the same as that of Theorem \ref{th:epsreg}. 
Since Lemma \ref{lem:epspre} still holds for $\overline{u}$ 
with $I_\rho[u]$ replaced by $I_\rho[\overline{u}]$, 
our temporal goal is to estimate $I_\rho[u_k]$ uniformly for $k$. 
After deriving the estimate, we will take a limit $k\to\infty$.

Let $0<\eps_*<1$ and $A>1$ be constants specified later. 
Assume \eqref{eq:1scalebl}. 
Proposition \ref{pro:stcon} shows that 
there exists $k_0$ satisfying 
\[
\begin{aligned}
	\delta^{\frac{4}{p-1}-n} 
	\int_{-\delta^2}^{-\delta^2/2} \int_{B_{A\delta}} |u_k|^{p+1} dxdt 
	\leq 2\eps_* 
\end{aligned}
\]
for all $k\geq k_0$. 
Therefore, by the same computations as in Lemma \ref{lem:32} 
based on Lemma \ref{lem:wp1Ek}, and by 
\eqref{eq:fIddef} and the first line in \eqref{eq:tilIdele}, 
we see that 
\[
\begin{aligned}
	&I_\rho[u_k] (\tilde x,\tilde t) 
	\leq 
	C h ( \eps_* + e^{-\frac{A^2}{C}} \cJ_{k,\delta} ),  \\
	&\cJ_{k,\delta}:= 
	\int_{-\log(2\delta^2)}^{-\log(3\delta^2/2)}  
	\int_{\R^n} |w_{k,(0,\delta^2)}|^{p+1} \rho  d\eta d\tau, 
\end{aligned}
\]
if $Q_\rho(\tilde x,\tilde t)\subset Q_{\delta/2}$, where 
$w_k$ is given in \eqref{eq:EkwkcEk}. 
By \eqref{eq:Ep1app}, we have 
\[
\begin{aligned}
	\cJ_{k,\delta} 
	&\leq 
	C( \cE_{k,(0,\delta^2)}(-\log(2\delta^2))
	-\cE_{k,(0,\delta^2)}(-\log(3\delta^2/2)) )^\frac{1}{2}  
	\cJ_{k,\delta}^\frac{1}{p+1} \\
	&\quad 
	+ \frac{2(p+1)}{p-1} \cE_{k,(0,\delta^2)}(-\log(2\delta^2)) 
	\log\frac{4}{3}. 
\end{aligned}
\]
For each $k\geq k_0$, there exists $k_0'\geq k$ such that 
$T_k \delta^2/2<T_{k_0'}<0$. 
Then, from \eqref{eq:EcEk}, the second equality in \eqref{eq:EkEtra} 
and the same computations as in Lemma \ref{lem:unibE} and \eqref{eq:unibElow}, 
it follows that 
\[
\begin{aligned}
	&\cE_{k,(0,\delta^2)}(-\log(2\delta^2)) 
	= E_{(0,-T_k\delta^2)}(T_k \delta^2) 
	\leq E_{(0,-T_k\delta^2)}(T_1)\leq C, \\
	&-\cE_{k,(0,\delta^2)}(-\log(3\delta^2/2)) 
	= -E_{(0,-T_k\delta^2)}(T_k \delta^2/2)  \leq -E_{(0,-T_k\delta^2)}(T_{k_0'})\leq C, 
\end{aligned}
\]
where $C>0$ depends only on $n$, $p$, $M$ and $\delta_1$.  
Hence $\cJ_{k,\delta}\leq C \cJ_{k,\delta}^{1/(p+1)}+C$, and so 
$\cJ_{k,\delta}\leq C$. 
Thus, 
$I_\rho[u_k](\tilde x,\tilde t) \leq  C h ( \eps_* + e^{-A^2/C} )$ holds 
if $Q_\rho(\tilde x,\tilde t)\subset Q_{\delta/2}$, 
where the right-hand side is independent of $k$ and $\delta$. 
Then by letting $k\to\infty$, 
Proposition \ref{pro:stcon} shows that 
\[
	I_\rho[\overline{u}] (\tilde x,\tilde t) 
	\leq 
	C h ( \eps_* + e^{-\frac{A^2}{C}} ) 
\]
if $Q_\rho(\tilde x,\tilde t)\subset Q_{\delta/2}$. 
We choose $\eps_*$ small and $A$ large. 
Then, Lemma \ref{lem:epspre} for $\overline{u}$ with 
$I_\rho[u]$ replaced by $I_\rho[\overline{u}]$ guarantees 
the desired estimate. 
\end{proof}

\section{Proof of main theorem}\label{sec:pr}
Let us prove Theorem \ref{th:main}. 
By a scaling and a time shift, we assume that 
$u$ satisfies \eqref{eq:fujita} 
and suppose the contradiction assumption \eqref{eq:seq}. 
We start from showing that blow-up at space infinity does not occur. 
Recall that 
$\delta_1$ and $\eps_0$ are given in 
\eqref{eq:cM} and Theorem \ref{th:epsreg}, respectively.

\begin{lemma}\label{lem:nobuin}
Suppose \eqref{eq:seq}. 
Let $0<\delta\leq \sqrt{-T_1}$. 
Then there exists a constant $R>1$ 
determined by $\delta$ and $\eps_0$ such that 
\[
	|u(x,t)|
	\leq C\delta^{-\frac{2}{p-1}}, \quad 
	x\in \R^n\setminus B_R, \; -\frac{\delta^2}{256}<t<0, 
\]
where $C>0$ is a constant depending only on $n$, $p$, $M$ and $\delta_1$. 
\end{lemma}

\begin{proof}
Let $0<\delta\leq \sqrt{-T_1}$. 
By \eqref{eq:uspace}, we have 
$u\in L^{q_c}( \R^n \times (-\delta^2, -\delta^2/2) )$. 
Then, for $\eps'>0$, there exists $R_{\eps'}>1$ such that 
\[
	\| u\|_{ L^{q_c}( (  (\R^n\setminus B_{R_{\eps'}}) 
	\times (-\delta^2, -\delta^2/2) )} 
	\leq \eps'. 
\]
Let $x_0$ satisfy $B_{A\delta}(x_0) \subset \R^n\setminus B_{R_{\eps'}}$. 
The H\"older inequality gives 
\[
	\delta^{\frac{4}{p-1}-n} 
	\int_{-\delta^2}^{-\delta^2/2} \int_{B_{A\delta}(x_0)} 
	|u|^{p+1} dxdt \leq 
	C \delta^{\frac{4}{p-1}-n+(n+2)(1-\frac{p+1}{q_c})} (\eps')^{p+1}, 
\]
where $C>0$ is independent of $x_0$. 
We fix $\eps'$ such that the right-hand side is smaller than $\eps_0$. 
Then, Theorem \ref{th:epsreg} guarantees 
$\|u\|_{L^\infty(Q_{\delta/16}(x_0,0)} 
\leq C\delta^{-2/(p-1)}$, 
where $C$ is independent of $x_0$. 
By setting $R=R_{\eps'}+A\delta+1$, the desired estimate holds. 
\end{proof}

In view of Lemma \ref{lem:nobuin}, without loss of generality, 
we can assume that the blow-up occurs at $(0,0)\in \R^{n+1}$. 
Let $\overline{u}$ be the blow-up limit given in Lemma \ref{lem:exbul}. 
Then, we show that 
an integral of $|\overline{u}|^{p+1}$ around the origin $(0,0)\in \R^{n+1}$
does not vanish. 
In particular, $\overline{u}\not\equiv 0$. 

\begin{lemma}\label{lem:noshrink}
Let $\eps_0$ be a constant given in Theorem \ref{th:epsreg}. Then, 
\[
	\int_{-1}^0 \int_{B_A}  |\overline{u}(x,t)|^{p+1} dx  dt 
	\geq \eps_0. 
\]
\end{lemma}

\begin{proof}
Let $0<c<1/2$. 
Since $u\not\in L^\infty(Q_{\delta/16})$,  
the contraposition of Theorem \ref{th:epsreg} gives 
\[
	\eps_0 < 
	\delta^{\frac{4}{p-1}-n} 
	\int_{-\delta^2}^{-c \delta^2} \int_{B_{A\delta}} |u(x,t)|^{p+1} dxdt 
\]
for any $0<\delta\leq \sqrt{-T_1}$. 
Substituting $\delta=\sqrt{-T_k}$ into the above inequality 
and changing the variables yield 
\[
	\int_{-1}^{-c} \int_{B_A} |u_k(y,s)|^{p+1} dyds > \eps_0 
\]
for $k\geq k_1$. 
By letting $k\to\infty$ with the aid of Proposition \ref{pro:stcon}, 
and then by letting $c\to 0$, the lemma follows. 
\end{proof}

In what follows, we prove $\overline{u}\equiv0$, 
contrary to Lemma \ref{lem:noshrink}. 
Since we suppose \eqref{eq:seq}, some subsequence still denoted by 
$u(\cdot,T_k)$ converges to a function weakly in $L^{q_c}(\R^n)$. 
We define $u(\cdot,0)$ and $u_k(\cdot,0)$ 
by this weak limit and \eqref{eq:ukdef} with $t=0$, respectively. 
Moreover, 
from the weak lower semicontinuity of the norm, it follows that 
\begin{equation}\label{eq:fibd}
	\|u_k(\cdot,0)\|_{L^{q_c}(\R^n)}
	= \|u(\cdot,0)\|_{L^{q_c}(\R^n)}
	\leq \liminf_{k\to\infty}\|u(\cdot,T_k)\|_{L^{q_c}(\R^n)} \leq M 
\end{equation}
for $k\geq1$. 
Thus, a subsequence $u_k(\cdot,0)$ converges to 
a function, denoted by $\overline{u}(\cdot,0)$, weakly in $L^{q_c}(\R^n)$. 
Then, $\overline{u}$ vanishes at the final time.

\begin{lemma}\label{lem:finbul}
$\overline{u}(\cdot,0)=0$ a.e.~in $\R^n$. 
\end{lemma}

\begin{proof}
For $R>0$, we take $\varphi\in C^\infty_0(\R^n)$ such that 
$\supp\varphi\subset B_R$. 
Since $u_k(\cdot,0)\to \overline{U}(\cdot,0)$ weakly in $L^{q_c}(\R^n)$, 
the first term in the right-hand side of 
\[
\begin{aligned}
	\int_{\R^n} \overline{u}(x,0) \varphi(x) dx 
	= \int_{\R^n} (\overline{u}(x,0)-u_k(x,0)) \varphi(x) dx 
	+\int_{\R^n} u_k(x,0) \varphi(x) dx
\end{aligned}
\]
converges to $0$ as $k\to\infty$. On the other hand, the change of variables, 
the H\"older inequality and \eqref{eq:fibd} yield 
\[
\begin{aligned}
	&\left| \int_{\R^n} u_k(x,0) \varphi(x) dx\right|  \leq  
	C \int_{B_R} |u_k(x,0)|  dx \\
	&= C (-T_k)^{\frac{1}{p-1}-\frac{n}{2}} \int_{B_{\sqrt{-T_k} R}} |u(y,0)| dy 
	\leq CR^{n-\frac{2}{p-1}} \|u(\cdot,0)\|_{L^{q_c}(B_{\sqrt{-T_k} R})} 
	\to 0
\end{aligned}
\]
as $k\to\infty$. 
Since $R$ is arbitrary, the lemma follows. 
\end{proof}

By Lemma \ref{lem:blepsreg}, 
we show the boundedness of $\overline{u}$ far from the singular point.

\begin{lemma}\label{lem:bddovu}
For any $0<\delta\leq \sqrt{-T_1}$, 
there exists a constant $R>1$ 
determined by $\delta$ and $\eps_0$ such that 
\[
	|\overline{u}(x,t)|
	\leq C\delta^{-\frac{2}{p-1}}
	\quad \mbox{ for a.e.~}
	x\in \R^n\setminus B_R, \; -\frac{\delta^2}{256}<t<0, 
\]
where $C>0$ is a constant depending only on 
$n$, $p$, $M$ and $\delta_1$. 
\end{lemma}

\begin{proof}
Define $\overline{U}(z):=\overline{u}(z,-1)$. 
By Lemma \ref{lem:BSS},  setting $\lambda = (-t)^{-1/2}$ gives 
$\overline{u}(x,t) = (-t)^{-\frac{1}{p-1}} \overline{U}(x/\sqrt{-t})$. 
From the scaling invariance of $L^{q_c}$ norm together with \eqref{eq:oubar1M}, 
it follows that 
\begin{equation}\label{eq:ubarbddM}
	\| \overline{u}\|_{L^\infty(-\infty,0; L^{q_c}(\R^n))}
	= \|\overline{U}\|_{L^{q_c}(\R^n)}\leq M 
\end{equation}
and 
$\overline{u}\in L^\infty(-\infty,0; L^{q_c}(\R^n))$. 
Thus,  $\overline{u}\in L^{q_c} ( \R^n\times (-\delta^2, -\delta^2/2))$ 
for $0<\delta\leq \sqrt{-T_1}$. 
By the same argument as in Lemma \ref{lem:nobuin} 
based on Lemma \ref{lem:blepsreg}, the desired bound follows. 
\end{proof}

By a similar argument of Wang \cite[Lemma 3.3]{Wa08} 
(see also \cite[Lemma 5.12]{MTpre}), 
we give a partial regularity for $\overline{u}$. 
For $t<0$, define 
\[
	\Sigma(t):= \{ x\in \R^n; \overline{u}(\cdot,t) \not\in L^\infty(B_\rho(x))
	\mbox{ for all }\rho>0 \}. 
\]
We denote by $\card(\Sigma(t))$ the cardinality of $\Sigma(t)$.

\begin{lemma}\label{lem:pareg}
There is $C>0$ depending only on 
$n$, $p$, $M$ and $\delta_1$ such that 
\[
	\card( \Sigma(t) ) \leq C
	\quad \mbox{ for any $t<0$}.
\]
\end{lemma}

\begin{proof}
Fix $t_0<0$. 
Without loss of generality, we can assume $\Sigma(t_0)\neq \emptyset$. 
Let $x_0\in \Sigma(t_0)$. 
Then Lemma \ref{lem:blepsreg} implies that 
\[
\begin{aligned}
	\eps_* &< \rho^{\frac{4}{p-1}-n} 
	\int_{t_0-\rho^2}^{t_0-\rho^2/2} \int_{B_{A\rho}(x_0)} 
	|\overline{u}|^{p+1} dxdt  \\
	&\leq 
	C \rho^{\frac{4}{p-1}-n + (n+2) (1-\frac{p+1}{q_c})} 
	\left( \int_{t_0-\rho^2}^{t_0-\rho^2/2} \int_{B_{A\rho}(x_0)} 
	|\overline{u}|^{q_c} dxdt \right)^\frac{p+1}{q_c} 
\end{aligned}
\]
for any $0<\rho\leq \sqrt{-T_1}$. Thus, 
\begin{equation}\label{eq:estdxdt}
	\eps_*^\frac{q_c}{p+1}  \leq 
	C \rho^{-2} \int_{t_0-\rho^2}^{t_0-\rho^2/2} \int_{B_{A\rho}(x_0)} 
	|\overline{u}|^{q_c} dxdt 
\end{equation}
for a constant $C>0$ depending only on $n$, $p$, $M$ and $\delta_1$.

Let $S=\{x_1,\ldots,x_N\}$ be any finite subset of $\Sigma(t_0)$ and 
let $0<\rho_S\leq \sqrt{-T_1}$ satisfy that 
the family $\{B_{A\rho_S}(x_i) \}_{1\leq i\leq N}\}$ is pairwise disjoint. 
Then by substituting $x_0=x_i$ and $\rho=\rho_S$ into \eqref{eq:estdxdt}, 
and then by summing up from $i=1$ to $i=N$ and 
using \eqref{eq:ubarbddM}, 
we see that 
\[
\begin{aligned}
	N\eps_*^\frac{q_c}{p+1} 
	&\leq 
	C \rho_S^{-2} \sum_{i=1}^N \int_{t_0-\rho_S^2}^{t_0-\rho_S^2/2} 
	\int_{B_{A\rho_S}(x_i)}  |\overline{u}|^{q_c} dx dt \\
	&\leq 
	C \rho_S^{-2} \int_{t_0-\rho_S^2}^{t_0-\rho_S^2/2} 
	\int_{\R^n}  |\overline{u}|^{q_c} dx dt 
	\leq  C M^{q_c}. 
\end{aligned}
\]
This yields $N\leq C M^{q_c} \eps_*^{-q_c/(p+1)}$, 
where $C$ depends only on $n$, $p$, $M$ and $\delta_1$. 
Therefore, $\card(\Sigma(t_0))\leq C M^{q_c} \eps_*^{-q_c/(p+1)}$. 
This completes the proof. 
\end{proof}

We are now in a position to prove Theorem \ref{th:main}.

\begin{proof}[Proof of Theorem \ref{th:main}]
To obtain a contradiction, we suppose \eqref{eq:seq}. 
By Lemma \ref{lem:nobuin}, without loss of generality, 
we can assume that the blow-up of $u$ occurs at 
$(0,0)\in \R^{n+1}$. 
Then, Lemma \ref{lem:noshrink} shows that 
the blow-up limit $\overline{u}$ must satisfy $\overline{u}\not \equiv 0$.

Based on Lemmas \ref{lem:finbul} and \ref{lem:bddovu}, 
we apply the backward uniqueness of 
Escauriaza, Seregin and \v{S}ver\'{a}k \cite[Theorem 5.1]{ESS03} 
to see that 
\[
	\overline{u}(x,t) \equiv 0 
	\quad \mbox{ for }
	x\in \R^n\setminus B_R, \; -\frac{\delta^2}{256}<t\leq 0. 
\]
Since $\overline{u}(x,-\delta^2/1024) 
= (\delta^2/1024)^{-1/(p-1)} \overline{U}(32x/\delta)$, 
we have $\overline{U}\equiv 0$ on $\R^n\setminus B_{32R/\delta}$. 
From direct computations and Lemma \ref{lem:pareg} with $t=-1$, 
it follows that $\overline{U}=\overline{U}(z)$ satisfies 
\[
	\Delta \overline{U} -\frac{1}{2}z \cdot \nabla \overline{U}
	-\frac{1}{p-1} \overline{U} +|\overline{U}|^{p-1}\overline{U} = 0
\]
in the classical sense except for the finite subset of $\R^n$. 
Then by applying the unique continuation for elliptic equations 
(see \cite{Ar57} for instance) in connected subsets of $\R^n$, 
we obtain $\overline{U}\equiv 0$ in $\R^n$. 
Hence we conclude that $\overline{u}\equiv 0$ on $\R^n\times(-\infty,0]$, 
a contradiction. 
The proof is complete. 
\end{proof}

\section*{Acknowledgments}
The first author was supported in part by JSPS KAKENHI 
Grant Numbers 17K05312, 21H00991 and 21H04433. 
The second author was supported in part 
by JSPS KAKENHI Grant Numbers 22H01131, 22KK0035 and 23K12998.

\end{document}